\documentclass{amsart}

\usepackage{amssymb,amsmath,amsthm,latexsym,qtree}

  \theoremstyle{definition}
  \newtheorem{definition}{Definition}
  \newtheorem{remark}[definition]{Remark}
  \newtheorem{example}[definition]{Example}

  \theoremstyle{plain}
  \newtheorem{lemma}[definition]{Lemma}
  \newtheorem{proposition}[definition]{Proposition}
  \newtheorem{theorem}[definition]{Theorem}
  \newtheorem{corollary}[definition]{Corollary}

  \newcommand{\sltwo}{\mathfrak{sl}_2(\mathbb{C})}

  \newcommand{\three}{\!\!\!}

\begin{document}

\title[Alternating quaternary algebra structures]
{Alternating quaternary algebra structures on irreducible representations of
$\sltwo$}

\author{Murray R. Bremner}

\address{Department of Mathematics and Statistics, University of Saskatchewan,
Canada}

\email{bremner@math.usask.ca}

\author{Hader A. Elgendy}

\address{Department of Mathematics and Statistics, University of Saskatchewan,
Canada}

\email{hae431@mail.usask.ca}

\date{\textit{\today}}

\begin{abstract}
We determine the multiplicity of the irreducible representation $V(n)$ of the
simple Lie algebra $\sltwo$ as a direct summand of its fourth exterior power
$\Lambda^4 V(n)$. The multiplicity is $1$ (resp.~2) if and only if $n = 4, 6$
(resp.~$n = 8, 10$).  For these $n$ we determine the multilinear polynomial
identities  of degree $ \le 7$ satisfied by the $\sltwo$-invariant alternating
quaternary algebra structures obtained from the projections $\Lambda^4 V(n) \to
V(n)$. We represent the polynomial identities as the nullspace of a large
integer matrix and use computational linear algebra to find the canonical basis
of the nullspace.
\end{abstract}

\maketitle


\section{Introduction}

The theory of $n$-ary generalizations of Lie algebras has been studied since
the late 1960's, beginning with fundamental work by Russian mathematicians; see
the survey articles by Kurosh \cite{Kurosh} and Baranovich and Burgin
\cite{BaranovichBurgin}. This theory also appeared naturally and independently
in various domains of theoretical physics. Indeed, the discovery of Nambu
mechanics \cite{Nambu} in 1973, as well as the work of Fillippov
\cite{Filippov1} in 1985, gave impulse to a significant development of this
theory. From a physical point of view, we mention the work of Takhtajan
\cite{Takhtajan}, Michor and Vinogradov \cite{MichorVinogradov}, de Azc\'arraga
and P\'erez-Bueno \cite{deAzcarragaPerezBueno}, Gautheron \cite{Gautheron},
Vaisman \cite{Vaisman}, Curtright and Zachos \cite{CurtrightZachos}, Curtright
et al.~\cite{Curtright+} and Ataguema et al.~\cite{Ataguema+}. From an
algebraic point of view, we mention the work of Kasymov \cite{Kasymov}, Ling
\cite{Ling}, Hanlon and Wachs \cite{HanlonWachs}, Gnedbaye \cite{Gnedbaye},
Bremner \cite{Bremner2}, Filippov \cite{Filippov2}, Bremner and Hentzel
\cite{BremnerHentzel1} and Pozhidaev \cite{Pozhidaev}. The latest development
in mathematical physics related to $n$-ary algebras is the influential work of
Bagger and Lambert \cite{BaggerLambert} and Gustavsson \cite{Gustavsson}, which
aims at a world-volume theory of multiple M2-branes. For a very recent
comprehensive survey of this entire area, referring to both the physical and
mathematical literature, see de Azc\'arraga and Izquierdo
\cite{deAzcarragaIzquierdo}.

The main difficulty in this theory is to find a useful generalization of the
Jacobi identity. There are two principal candidates: (1) the derivation (or
Filippov) identity, which states that the multiplication operators in the
algebra are derivations of the $n$-ary structure, and (2) the alternating sum
identity, which states that the alternating sum over all possible nested pairs
of operations is identically zero. For $n = 2$ both of these identities reduce
to the familiar Jacobi identity for Lie algebras.

A serious disadvantage of the theory of Filippov algebras ($n$-Lie algebras) is
that for $n \ge 3$ it has been shown by Ling \cite{Ling} that there is only one
simple finite-dimensional object over an algebraically closed field of
characteristic 0. That is, Filippov algebras generalize the 3-dimensional
simple Lie algebra to the $n$-ary case for $n \ge 3$, but not any of the other
simple Lie algebras. It seems that the definition of Filippov algebra is too
restrictive: the derivation identity is too strong. From this point of view,
the alternating sum identity is a good candidate for a weaker identity which
can still be naturally regarded as a generalization of the Jacobi identity. It
is an open problem to classify the alternating $n$-ary algebras which satisfy
the alternating sum identity, but it is clear from the results in Bremner and
Hentzel \cite{BremnerHentzel2} and the present paper that there is more than
one simple object, at least for $n = 3$ and $n = 4$.

Our approach is to use the representation theory of Lie algebras to construct
new alternating algebra structures and to discover natural $n$-ary
generalizations of Lie algebras. In particular, we study the exterior powers of
an irreducible representation of a simple Lie algebra; these are important
objects in invariant theory, algebraic geometry, and the theory of Lie groups
(see for example Fulton and Harris \cite{FultonHarris}). If $\Lambda^n V$, the
$n$-th exterior power of an irreducible representation $V$ of a simple Lie
algebra $L$, contains $V$ itself as a direct summand, then the projection
$\Lambda^n V \to V$ defines an alternating $n$-ary algebra structure on $V$
which is $L$-invariant in the sense that the derivation algebra of this $n$-ary
structure contains a subalgebra isomorphic to $L$. This approach was initiated
by Bremner and Hentzel \cite{BremnerHentzel2} in the case of the third exterior
powers of irreducible representations of the 3-dimensional simple Lie algebra
$\mathfrak{sl}_2(\mathbb{C})$. In this paper we extend this work to the fourth
exterior power.

We recall some basic information about representation theory of Lie algebras.
The simple Lie algebra $\sltwo$ has basis $\{ H, E, F \}$ and structure
constants
  \[
  [H,E] = 2E, \qquad [H,F] = -2F, \qquad [E,F] = H.
  \]
All other brackets follow from bilinearity and anticommutativity. For $n \in
\mathbb{Z}$, $n \ge 0$, the irreducible representation $V(n)$ of $\sltwo$ with
highest weight $n$ has dimension $n+1$; the action of $\sltwo$ with respect to
the basis $\{ \, v_{n-2i} \, | \, i=0,\hdots,n \, \}$ is
  \allowdisplaybreaks
  \begin{align}
  &
  H.v_{n-2i} = (n{-}2i) v_{n-2i},
  \label{Haction}
  \\
  &
  E.v_n = 0,
  \qquad
  E.v_{n-2i} = (n{-}i{+}1) v_{n-2i+2}
  \;
  (i=1,\hdots,n),
  \label{Eaction}
  \\
  &
  F.v_{n-2i} = (i{+}1) v_{n-2i-2}
  \;
  (i=0,\hdots,n-1) \, ,
  \qquad
  F.v_{-n} = 0.
  \label{Faction}
  \end{align}
Any finite dimensional irreducible representation of $\sltwo$ is isomorphic to
$V(n)$ for some $n$. Any finite dimensional representation of $\sltwo$ is
isomorphic to a direct sum of irreducible representations. The multiplicity of
$V(n)$ in its $k$-th exterior power $\Lambda^k V(n)$ is the dimension of the
vector space
  \[
  \mathrm{Hom}_{\sltwo}\big( \Lambda^k V(n), V(n) \big)
  \]
of $\sltwo$-invariant linear maps $P \colon \Lambda^k V(n) \to V(n)$. If this
multiplicity is positive then $P$ defines an alternating $k$-ary algebra
structure on $V(n)$,
  \[
  [x_1,\hdots,x_k] = P( x_1 \wedge \cdots \wedge x_k ),
   \]
which is $\sltwo$-invariant in the sense that the action of any $L \in \sltwo$
is a derivation of the $k$-ary multiplication: for any $x_1,\hdots,x_k \in
V(n)$ we have
  \[
  L.[x_1,\hdots,x_i,\hdots,x_k]
  =
  \sum_{i=1}^k
  [ x_1, \hdots, L.x_i, \hdots, x_k ],
  \]
For the representation theory of $\sltwo$ we refer to Humphreys
\cite{Humphreys}.


\section{Multiplicity formulas}

Bremner and Hentzel \cite{BremnerHentzel1} studied the case $k = 2$
corresponding to alternating binary algebra structures on $V(n)$. We have
  \[
  \dim \mathrm{Hom}_{\sltwo}\big( \Lambda^2 V(n), V(n) \big)
  =
  \begin{cases}
  1 &\text{if $n \equiv 2$ (mod 4)}, \\
  0 &\text{otherwise}.
  \end{cases}
  \]
In this case $n = 2$ gives the 3-dimensional adjoint representation of
$\sltwo$, $n = 6$ gives the 7-dimensional simple non-Lie Malcev algebra, and $n
= 10$ gives a new 11-dimensional anticommutative algebra satisfying a
polynomial identity of degree 7. Bremner and Hentzel \cite{BremnerHentzel2}
studied the case $k = 3$ corresponding to alternating ternary algebra
structures on $V(n)$. For $n = 6q + r$ ($0 \le r\le 5$) we have
  \[
  \dim \mathrm{Hom}_{\sltwo}\big( \Lambda^3 V(n), V(n) \big)
  =
  \begin{cases}
  q &\text{if $r = 0, 1, 2, 4$}, \\
  q+1 &\text{if $r = 3, 5$}.
  \end{cases}
  \]
The multiplicity is 1 for $n = 3, 5, 6, 7, 8, 10$; the corresponding $V(n)$
provide new examples of alternating ternary algebras.

In this paper we consider the case $k = 4$: we study alternating quaternary
algebra structures on $V(n)$ obtained from $\sltwo$-invariant linear maps
$\Lambda^4 V(n) \to V(n)$.  We first obtain a closed formula for the
multiplicity $\dim \mathrm{Hom}_{\sltwo} \big( \Lambda^4 V(n), V(n) \big)$
using a general approach which applies to arbitrary exterior powers.

\begin{theorem} \label{multiplicitytheorem}
If $n$ is odd then $\dim \mathrm{Hom}_{\sltwo}\big( \Lambda^4 V(n), V(n) \big)
= 0$. If $n$ is even then $n = 24 q + r$ with $0 \le r < 24$ ($r$ even) and we
have
  \begin{align*}
  &
  \dim \mathrm{Hom}_{\sltwo}\big( \Lambda^4 V(n), V(n) \big) =
  \\
  &
  \frac{1}{1152}
  \left\{\three
  \begin{array}{ll|ll}
  30n^2+96n &\text{if} \;\; r = 0, 16 &
  30n^2-120n+120 &\text{if} \;\; r = 2 \\
  30n^2+96n+288 &\text{if} \;\; r = 4, 12 &
  30n^2-120n+792 &\text{if} \;\; r = 6, 22 \\
  30n^2+96n-384 &\text{if} \;\; r = 8 &
  30n^2-120n+504 &\text{if} \;\; r = 10, 18 \\
  30n^2-120n+408 &\text{if} \;\; r = 14 &
  30n^2+96n-96 &\text{if} \;\; r = 20
  \end{array}
  \right.
  \end{align*}
\end{theorem}

\begin{proof}
A proof using P\'olya enumeration is given in Section \ref{proofsection}.
\end{proof}

\begin{example}
The multiplicities for $n = 24q + r$ ($0 \le q \le 9$) are given in Table
\ref{multiplicityexample}.
\end{example}

\begin{corollary}
The representation $V(n)$ of $\sltwo$ occurs in $\Lambda^4 V(n)$ with
multiplicity 1 (resp.~2) if and only if $n = 4$ or $n = 6$ (resp.~$n = 8$ or $n
= 10$).
\end{corollary}

\begin{proof}
The vertices of the parabolas in Theorem \ref{multiplicitytheorem} occur at
either $n = -8/5$ or $n = 2$, so for each $r$ the multiplicity is an increasing
function of $q$.
\end{proof}

For $n = 4, 6, 8, 10$ we use computational linear algebra to find all the
multilinear polynomial identities of degree $\le 7$ satisfied by the resulting
quaternary algebras.


\section{Quaternary algebra structures}

In this section we explain how to compute explicitly the decomposition of
$\Lambda^4 V(n)$ as a direct sum of irreducible representations of $\sltwo$,
together with an explicit multiplication table for the alternating quaternary
algebra structure on $V(n)$ obtained from a projection $\Lambda^4 V(n) \to
V(n)$. Recall that $V(n)$ has the vector space basis $\{ \, v_{n-2i} \mid i =
0, 1, \hdots, n \, \}$ and that the subscript on $v_{n-2i}$ is its weight: its
eigenvalue for the action of $H \in \sltwo$.

\begin{lemma} \label{multiplicityformula} \cite[Lemma 5.1]{BremnerHentzel2}
Let $M$  be an $\sltwo$-module with $\dim M < \infty$. For $n \in \mathbb{Z}$
let $M_n = \{ v \in M \, | \, H.v = nv \}$ be the subspace of all vectors of
weight $n$ together with $0$. For $n \ge 0$ the multiplicity of $V(n)$ in the
decomposition of $M$ as a direct sum of simple $\sltwo$-modules is $\dim M_n -
\dim M_{n+2}$.
\end{lemma}

\begin{definition} \label{definitiontensorbasis}
The \textbf{tensor basis} of $\Lambda^4 V(n)$ consists of $\binom{n{+}1}{4}$
\textbf{quadruples}:
  \[
  v_p \wedge v_q \wedge v_r \wedge v_s
  =
  \sum_{\sigma \in S_4}
  \epsilon(\sigma)
  \big(
  v_{\sigma(p)} \otimes v_{\sigma(q)} \otimes v_{\sigma(r)} \otimes v_{\sigma(s)}
  \big),
  \]
where $n \ge p > q > r > s \ge -n$ with $p,q,r,s \equiv n \, (\mathrm{mod}\,2)$
and $\epsilon\colon S_4 \to \{\pm1\}$ is the sign homomorphism. We usually
abbreviate $v_p \wedge v_q \wedge v_r \wedge v_s$ by $[p,q,r,s]$.  The action
of $L \in \sltwo$ satisfies the derivation property,
  \begin{equation} \label{commutingformula}
  \begin{array}{r}
  L . ( v_p \wedge v_q \wedge v_r \wedge v_s )
  =
  L . v_p \wedge v_q \wedge v_r \wedge v_s
  +
  v_p \wedge L . v_q \wedge v_r \wedge v_s
  \qquad
  \\
  {}
  +
  v_p \wedge v_q \wedge L . v_r \wedge v_s
  +
  v_p \wedge v_q \wedge v_r \wedge L . v_s,
  \end{array}
  \end{equation}
and hence the \textbf{weight} of the quadruple $T = [p,q,r,s]$ is $w(T) =
p+q+r+s$. The \textbf{standard order} of the quadruples is given by decreasing
weight, and within each weight by reverse lex order: $T = [p,q,r,s]$ precedes
$T' = [p',q',r',s']$ if and only if either $w(T) > w(T')$ or $w(T) = w(T')$ and
$t > t'$ where $t, t'$ are the components of $T, T'$ in the leftmost position
where the components are not equal.
\end{definition}

\begin{example} \label{tensorbasis}
The quadruples in the tensor basis of $\Lambda^4 V(n)$ for $n = 4, 6, 8$ are
displayed in Tables \ref{tensorweightvectorbasistable4},
\ref{tensorbasistable6}, \ref{tensorbasistable8} in standard order.
\end{example}

\begin{remark} \label{completedecomposition}
If we apply Lemma \ref{multiplicityformula} to Tables
\ref{tensorweightvectorbasistable4}, \ref{tensorbasistable6},
\ref{tensorbasistable8} then we obtain the decomposition of $\Lambda^4 V(n)$
for $n = 4, 6, 8$ as a direct sum of irreducible representations:
  \begin{align*}
  \Lambda^4 V(4)
  &\cong
  V(4),
  \qquad
  \Lambda^4 V(6)
  \cong
  V(12) \oplus V(8) \oplus V(6) \oplus V(4) \oplus V(0).
  \\
  \Lambda^4 V(8)
  &\cong
  V(20) {\oplus}
  V(16) {\oplus}
  V(14) {\oplus}
  2 V(12) {\oplus}
  V(10) {\oplus}
  2 V(8) {\oplus}
  V(6) {\oplus}
  2 V(4)
  {\oplus}
  V(0).
  \end{align*}
Similar computations establish the decomposition of $\Lambda^4 V(10)$:
  \begin{align*}
  \Lambda^4 V(10)
  &\cong
  V(28) \oplus
  V(24) \oplus
  V(22) \oplus
  2 V(20) \oplus
  V(18) \oplus
  3 V(16) \oplus
  2 V(14)
  \\
  &\quad
  \oplus
  3 V(12) \oplus
  2 V(10) \oplus
  3 V(8) \oplus
  V(6) \oplus
  3 V(4) \oplus
  V(0).
  \end{align*}
\end{remark}

The next step is to determine the highest weight vectors for the irreducible
summands of $\Lambda^4 V(n)$ as linear combinations of the quadruples in the
tensor basis.

\begin{lemma} \label{topvector}
The quadruple $[n,n{-}2,n{-}4,n{-}6]$ is the quadruple with highest weight in
$\Lambda^4 V(n)$ and is a highest weight vector for the summand $V(4n{-}12)$.
\end{lemma}

\begin{proof}
This follows directly from equations \eqref{Eaction} and
\eqref{commutingformula}.
\end{proof}

\begin{example}
For $n = 4$ we have $\Lambda^4 V(4) \cong V(4)$, and so the quadruple
$[4,2,0,-2]$ is the only highest weight vector in $\Lambda^4 V(4)$. If we
identify $[4,2,0,-2]$ with the highest weight vector $v_4$ of $V(4)$, and
repeatedly apply $F$ using equations \eqref{Faction} and
\eqref{commutingformula}, then we obtain the weight vectors of $\Lambda^4 V(4)$
corresponding to the basis vectors $v_2$, $v_0$, $v_{-2}$, $v_{-4}$ of $V(4)$:
  \begin{align*}
  &
  v_4 = [4,2,0,-2], \quad
  v_2 = F . v_4 = 4 [4,2,0,-4 ], \quad
  v_0 = \frac{1}{2!} F^2 . v_4 = 6 [4,2,-2,-4 ],
  \\
  &
  v_{-2} = \frac{1}{3!} F^3 . v_4 = 4 [4,0,-2,-4 ], \quad
  v_{-4} = \frac{1}{4!} F^4 . v_4 = [2,0,-2,-4 ].
  \end{align*}
The matrix expressing the weight vectors in $V(4)$ in terms of the quadruples
in $\Lambda^4 V(4)$ is $C = \mathrm{diag}(1,4,6,4,1)$. The matrix expressing
the quadruples in $\Lambda^4 V(4)$ in terms of the weight vectors in $V(4)$ is
$C^{-1} = \mathrm{diag}(1,\frac14,\frac16,\frac14,1)$. We now have the
structure constants for the $\sltwo$-invariant alternating quaternary algebra
structure on $V(4)$, which we denote by $[v_p,v_q,v_r,v_s]$:
  \begin{alignat*}{3}
  [v_4,v_2,v_0,v_{-2}] &= v_4,
  &\quad
  [v_4,v_2,v_0,v_{-4}] &= \tfrac14 v_4,
  &\quad
  [v_4,v_2,v_{-2},v_{-4}] &= \tfrac16 v_4,
  \\
  [v_4,v_0,v_{-2},v_{-4}] &= \tfrac14 v_4,
  &\quad
  [v_2,v_0,v_{-2},v_{-4}] &= v_4.
  \end{alignat*}
The LCM of the denominators of the coefficients is 12. Taking $a = \root 3 \of
{12}$ and setting $v'_t = v_t/a$, we obtain integral structure constants (see
Table \ref{structureconstants4}):
  \begin{alignat*}{3}
  [v'_4,v'_2,v'_0,v'_{-2}] &= 12 v'_4,
  &\quad
  [v'_4,v'_2,v'_0,v'_{-4}] &= 3 v'_4,
  &\quad
  [v'_4,v'_2,v'_{-2},v'_{-4}] &= 2 v'_4,
  \\
  [v'_4,v'_0,v'_{-2},v'_{-4}] &= 3 v'_4,
  &\quad
  [v'_2,v'_0,v'_{-2},v'_{-4}] &= 12 v'_4.
  \end{alignat*}
\end{example}

In general, for all other weights $w < 4n{-}12$ we need to find a basis for the
subspace of highest weight vectors of weight $w$ in $\Lambda^4 V(n)$. The
dimension of this subspace is the multiplicity of $V(w)$ as a summand of
$\Lambda^4 V(n)$.

\begin{definition} \label{definitionEmatrix}
Suppose that $4n{-}14 \ge w \ge 0$ ($w$ even). Let $d(w)$ be the dimension of
the weight space of weight $w$ in $\Lambda^4 V(n)$: the number of quadruples of
weight $w$. We define the matrix $E^{(n)}_w$ of size $d(w{+}2) \times d(w)$ by
setting the $(i,j)$ entry equal to the coefficient of the $i$-th quadruple of
weight $w{+}2$ in the expression for the action of $E \in \sltwo$ on the $j$-th
quadruple of weight $w$.  We call this the \textbf{E-action matrix} for weight
$w$ of $\Lambda^4 V(n)$; the nonzero vectors in its nullspace are the highest
weight vectors of weight $w$ in $\Lambda^4 V(n)$. We compute the row canonical
form (RCF) and extract the canonical integral basis (CIB) by setting the free
variables equal to the standard unit vectors, clearing denominators, and
canceling common factors.
\end{definition}

\begin{example} \label{highestweightvectors6}
For $\Lambda^4 V(6)$ we use the weight space basis of Table
\ref{tensorbasistable6} and obtain
  \allowdisplaybreaks
  \begin{align*}
  E^{(6)}_8
  &=
  \left[ \begin{array}{rr}
  2 & 4
  \end{array} \right]
  \xrightarrow{\mathrm{RCF}}
  \left[ \begin{array}{rr}
  1 & 2
  \end{array} \right]
  \xrightarrow{\mathrm{CIB}}
  \left[ \begin{array}{rr}
  -2 & 1
  \end{array} \right]
  \\[2pt]
  E^{(6)}_6
  &=
  \left[ \begin{array}{rrr}
  1 &\!\! 4 &\!\! . \\
  . &\!\! 2 &\!\! 5
  \end{array} \right]
  \xrightarrow{\mathrm{RCF}}
  \left[ \begin{array}{rrr}
  1 &\!\! . &\!\! -10 \\
  . &\!\! 1 &\!\! 5/2
  \end{array} \right]
  \xrightarrow{\mathrm{CIB}}
  \left[ \begin{array}{rrr}
  20 &\!\! -5 &\!\! 2
  \end{array} \right]
  \\[2pt]
  E^{(6)}_4
  &=
  \left[ \begin{array}{rrrr}
  4 &\!\! . &\!\! . &\!\! . \\
  1 &\!\! 3 &\!\! 5 &\!\! . \\
  . &\!\! . &\!\! 2 &\!\! 6
  \end{array} \right]
  \xrightarrow{\mathrm{RCF}}
  \left[ \begin{array}{rrrr}
  1 &\!\! . &\!\! . &\!\! . \\
  . &\!\! 1 &\!\! . &\!\! -5 \\
  . &\!\! . &\!\! 1 &\!\! 3
  \end{array} \right]
  \xrightarrow{\mathrm{CIB}}
  \left[ \begin{array}{rrrr}
  . &\!\! 5 &\!\! -3 &\!\! 1
  \end{array} \right]
  \\[2pt]
  E^{(6)}_0
  &=
  \left[ \begin{array}{rrrrr}
  2 &\!\! 5 &\!\! . &\!\! . &\!\! . \\
  . &\!\! 3 &\!\! . &\!\! 6 &\!\! . \\
  . &\!\! 1 &\!\! 4 &\!\! . &\!\! 6 \\
  . &\!\! . &\!\! . &\!\! 1 &\!\! 3
  \end{array} \right]
  \xrightarrow{\mathrm{RCF}}
  \left[ \begin{array}{rrrrr}
  1 &\!\! . &\!\! . &\!\! . &\!\! 15 \\
  . &\!\! 1 &\!\! . &\!\! . &\!\! -6 \\
  . &\!\! . &\!\! 1 &\!\! . &\!\! 3 \\
  . &\!\! . &\!\! . &\!\! 1 &\!\! 3
  \end{array} \right]
  \xrightarrow{\mathrm{CIB}}
  \left[ \begin{array}{rrrrr}
  -15 &\!\! 6 &\!\! -3 &\!\! -3 &\!\! 1
  \end{array} \right]
  \end{align*}
\end{example}

\begin{example} \label{highestweightvectors8}
For $\Lambda^4 V(8)$ we use the weight space basis of Table
\ref{tensorbasistable8} and obtain
  \allowdisplaybreaks
  \begin{align*}
  E^{(8)}_{16}
  &=
  \left[ \begin{array}{rr}
  4 &\three 6
  \end{array} \right]
  \xrightarrow{\mathrm{RCF}}
  \left[ \begin{array}{rr}
  1 &\three 3/2
  \end{array} \right]
  \xrightarrow{\mathrm{CIB}}
  \left[ \begin{array}{rr}
  -3 &\three 2
  \end{array} \right]
  \\[4pt]
  E^{(8)}_{14}
  &=
  \left[ \begin{array}{rrr}
  3 &\three 6 &\three . \\
  . &\three 4 &\three 7
  \end{array} \right]
  \xrightarrow{\mathrm{RCF}}
  \left[ \begin{array}{rrr}
  1 &\three . &\three -7/2 \\
  . &\three 1 &\three 7/4
  \end{array} \right]
  \xrightarrow{\mathrm{CIB}}
  \left[ \begin{array}{rrr}
  14 &\three -7 &\three 4
  \end{array} \right]
  \\[4pt]
  E^{(8)}_{12}
  &=
  \left[ \begin{array}{rrrrr}
  2 &\three 6 &\three . &\three . &\three . \\
  . &\three 3 &\three 5 &\three 7 &\three . \\
  . &\three . &\three . &\three 4 &\three 8
  \end{array} \right]
  \xrightarrow{\mathrm{RCF}}
  \left[ \begin{array}{rrrrr}
  1 &\three . &\three -5 &\three  . &\three 14 \\
  . &\three 1 &\three 5/3 &\three . &\three -14/3 \\
  . &\three . &\three . &\three 1 &\three 2
  \end{array} \right]
  \xrightarrow{\mathrm{CIB}}
  \left[ \begin{array}{rrrrr}
   15 &\three  -5 &\three 3 &\three  . &\three . \\
  -42 &\three 14 &\three . &\three -6 &\three 3
  \end{array} \right]
  \\[4pt]
  E^{(8)}_{10}
  &=
  \left[ \begin{array}{rrrrrr}
  1 &\three 6 &\three . &\three . &\three . &\three . \\
  . &\three 2 &\three 5 &\three 7 &\three . &\three . \\
  . &\three . &\three 3 &\three . &\three 7 &\three . \\
  . &\three . &\three . &\three 3 &\three 5 &\three 8 \\
  . &\three . &\three . &\three . &\three . &\three 4
  \end{array} \right]
  \xrightarrow{\mathrm{RCF}}
  \left[ \begin{array}{rrrrrr}
  1 &\three . &\three . &\three . &\three 70 &\three . \\
  . &\three 1 &\three . &\three . &\three -35/3 &\three . \\
  . &\three . &\three 1 &\three . &\three 7/3 &\three . \\
  . &\three . &\three . &\three 1 &\three 5/3 &\three . \\
  . &\three . &\three . &\three . &\three . &\three 1
  \end{array} \right]
  \\
  &\qquad
  \xrightarrow{\mathrm{CIB}}
  \left[ \begin{array}{rrrrrr}
  -210 &\three 35 &\three -7 &\three -5 &\three 3 &\three .   \end{array} \right]
  \\[4pt]
  E^{(8)}_8
  &=
  \left[ \begin{array}{rrrrrrrr}
  6 &\three . &\three . &\three . &\three . &\three . &\three . &\three . \\
  1 &\three 5 &\three . &\three 7 &\three . &\three . &\three . &\three . \\
  . &\three 2 &\three 4 &\three . &\three 7 &\three . &\three . &\three . \\
  . &\three . &\three . &\three 2 &\three 5 &\three . &\three 8 &\three . \\
  . &\three . &\three . &\three . &\three 3 &\three 6 &\three . &\three 8 \\
  . &\three . &\three . &\three . &\three . &\three . &\three 3 &\three 5
  \end{array} \right]
  \xrightarrow{\mathrm{RCF}}
  \left[ \begin{array}{rrrrrrrr}
  1 &\three . &\three . &\three . &\three . &\three . &\three . &\three . \\
  . &\three 1 &\three . &\three . &\three . &\three 7 &\three . &\three 56/3 \\
  . &\three . &\three 1 &\three . &\three . &\three -7 &\three . &\three -14 \\
  . &\three . &\three . &\three 1 &\three . &\three -5 &\three . &\three -40/3 \\
  . &\three . &\three . &\three . &\three 1 &\three 2 &\three . &\three 8/3 \\
  . &\three . &\three . &\three . &\three . &\three . &\three 1 &\three 5/3
  \end{array} \right]
  \\
  &\qquad
  \xrightarrow{\mathrm{CIB}}
  \left[ \begin{array}{rrrrrrrr}
  . &\three -7 &\three 7 &\three 5 &\three -2 &\three 1 &\three . &\three . \\
  . &\three -56 &\three 42 &\three 40 &\three -8 &\three . &\three -5 &\three 3
  \end{array} \right]
  \\[4pt]
  E^{(8)}_6
  &=
  \left[ \begin{array}{rrrrrrrrr}
  5 &\three . &\three 7 &\three . &\three . &\three . &\three . &\three . &\three . \\
  1 &\three 4 &\three . &\three 7 &\three . &\three . &\three . &\three . &\three . \\
  . &\three 2 &\three . &\three . &\three 7 &\three . &\three . &\three . &\three . \\
  . &\three . &\three 1 &\three 5 &\three . &\three . &\three 8 &\three . &\three . \\
  . &\three . &\three . &\three 2 &\three 4 &\three 6 &\three . &\three 8 &\three . \\
  . &\three . &\three . &\three . &\three . &\three 3 &\three . &\three . &\three 8 \\
  . &\three . &\three . &\three . &\three . &\three . &\three 2 &\three 5 &\three . \\
  . &\three . &\three . &\three . &\three . &\three . &\three . &\three 3 &\three 6
  \end{array} \right]
  \xrightarrow{\mathrm{RCF}}
  \left[ \begin{array}{rrrrrrrrr}
  1 &\three . &\three . &\three . &\three . &\three . &\three . &\three . &\three -224/3 \\
  . &\three 1 &\three . &\three . &\three . &\three . &\three . &\three . &\three 70/3 \\
  . &\three . &\three 1 &\three . &\three . &\three . &\three . &\three . &\three 160/3 \\
  . &\three . &\three . &\three 1 &\three . &\three . &\three . &\three . &\three -8/3 \\
  . &\three . &\three . &\three . &\three 1 &\three . &\three . &\three . &\three -20/3 \\
  . &\three . &\three . &\three . &\three . &\three 1 &\three . &\three . &\three 8/3 \\
  . &\three . &\three . &\three . &\three . &\three . &\three 1 &\three . &\three -5 \\
  . &\three . &\three . &\three . &\three . &\three . &\three . &\three 1 &\three 2
  \end{array} \right]
  \\
  &\qquad
  \xrightarrow{\mathrm{CIB}}
  \left[ \begin{array}{rrrrrrrrr}
  224 &\three -70 &\three -160 &\three 8 &\three 20 &\three -8 &\three 15 &\three -6
  &\three 3
  \end{array} \right]
  \\[4pt]
  E^{(8)}_4
  &=
  \left[ \begin{array}{rrrrrrrrrrr}
  4 &\three . &\three 7 &\three . &\three . &\three . &\three . &\three . &\three . &\three .
  &\three . \\
  1 &\three 3 &\three . &\three 7 &\three . &\three . &\three . &\three . &\three . &\three .
  &\three . \\
  . &\three . &\three 5 &\three . &\three . &\three . &\three 8 &\three . &\three . &\three .
  &\three . \\
  . &\three . &\three 1 &\three 4 &\three 6 &\three . &\three . &\three 8 &\three . &\three .
  &\three . \\
  . &\three . &\three . &\three 2 &\three . &\three 6 &\three . &\three . &\three 8 &\three .
  &\three . \\
  . &\three . &\three . &\three . &\three 2 &\three 4 &\three . &\three . &\three . &\three 8
  &\three . \\
  . &\three . &\three . &\three . &\three . &\three . &\three 1 &\three 5 &\three . &\three .
  &\three . \\
  . &\three . &\three . &\three . &\three . &\three . &\three . &\three 2 &\three 4 &\three 6
  &\three . \\
  . &\three . &\three . &\three . &\three . &\three . &\three . &\three . &\three . &\three 3
  &\three 7
  \end{array} \right]
  \xrightarrow{\mathrm{RCF}}
  \left[ \begin{array}{rrrrrrrrrrr}
  1 &\three . &\three . &\three . &\three . &\three . &\three . &\three . &\three -28 &\three .
  &\three 98 \\
  . &\three 1 &\three . &\three . &\three . &\three . &\three . &\three . &\three 14 &\three .
  &\three -245/3 \\
  . &\three . &\three 1 &\three . &\three . &\three . &\three . &\three . &\three 16 &\three .
  &\three -56 \\
  . &\three . &\three . &\three 1 &\three . &\three . &\three . &\three . &\three -2 &\three .
  &\three 21 \\
  . &\three . &\three . &\three . &\three 1 &\three . &\three . &\three . &\three -4 &\three .
  &\three 14/3 \\
  . &\three . &\three . &\three . &\three . &\three 1 &\three . &\three . &\three 2 &\three .
  &\three -7 \\
  . &\three . &\three . &\three . &\three . &\three . &\three 1 &\three . &\three -10 &\three .
  &\three 35 \\
  . &\three . &\three . &\three . &\three . &\three . &\three . &\three 1 &\three 2 &\three .
  &\three -7 \\
  . &\three . &\three . &\three . &\three . &\three . &\three . &\three . &\three . &\three 1
  &\three 7/3 \\
  \end{array} \right]
  \\
  &\qquad
  \xrightarrow{\mathrm{CIB}}
  \left[ \begin{array}{rrrrrrrrrrr}
  28 &\three -14 &\three -16 &\three 2 &\three 4 &\three -2 &\three 10 &\three -2
  &\three 1 &\three . &\three . \\
  -294 &\three 245 &\three 168 &\three -63 &\three -14 &\three 21 &\three -105
  &\three 21 &\three . &\three -7 &\three 3
  \end{array} \right]
  \\[4pt]
  E^{(8)}_0
  &=
  \left[ \begin{array}{rrrrrrrrrrrr}
  2 &\three 7 &\three . &\three . &\three . &\three . &\three . &\three . &\three . &\three .
  &\three . &\three . \\
  . &\three 3 &\three 6 &\three . &\three . &\three 8 &\three . &\three . &\three . &\three .
  &\three . &\three . \\
  . &\three 1 &\three . &\three 6 &\three . &\three . &\three 8 &\three . &\three . &\three .
  &\three . &\three . \\
  . &\three . &\three 4 &\three . &\three . &\three . &\three . &\three 8 &\three . &\three .
  &\three . &\three . \\
  . &\three . &\three 1 &\three 3 &\three 5 &\three . &\three . &\three . &\three 8 &\three .
  &\three . &\three . \\
  . &\three . &\three . &\three . &\three 2 &\three . &\three . &\three . &\three . &\three 8
  &\three . &\three . \\
  . &\three . &\three . &\three . &\three . &\three 4 &\three . &\three 6 &\three . &\three .
  &\three . &\three . \\
  . &\three . &\three . &\three . &\three . &\three 1 &\three 3 &\three . &\three 6 &\three .
  &\three . &\three . \\
  . &\three . &\three . &\three . &\three . &\three . &\three . &\three 1 &\three 4 &\three .
  &\three 7 &\three . \\
  . &\three . &\three . &\three . &\three . &\three . &\three . &\three . &\three 2 &\three 5
  &\three . &\three 7 \\
  . &\three . &\three . &\three . &\three . &\three . &\three . &\three . &\three . &\three .
  &\three 2 &\three 4
  \end{array} \right]
  \xrightarrow{\mathrm{RCF}}
  \left[ \begin{array}{rrrrrrrrrrrr}
  1 &\three . &\three . &\three . &\three . &\three . &\three . &\three . &\three . &\three .
  &\three . &\three   224 \\
  . &\three 1 &\three . &\three . &\three . &\three . &\three . &\three . &\three . &\three .
  &\three . &\three   -64 \\
  . &\three . &\three 1 &\three . &\three . &\three . &\three . &\three . &\three . &\three .
  &\three . &\three    16 \\
  . &\three . &\three . &\three 1 &\three . &\three . &\three . &\three . &\three . &\three .
  &\three . &\three    12 \\
  . &\three . &\three . &\three . &\three 1 &\three . &\three . &\three . &\three . &\three .
  &\three . &\three    -8 \\
  . &\three . &\three . &\three . &\three . &\three 1 &\three . &\three . &\three . &\three .
  &\three . &\three    12 \\
  . &\three . &\three . &\three . &\three . &\three . &\three 1 &\three . &\three . &\three .
  &\three . &\three    -1 \\
  . &\three . &\three . &\three . &\three . &\three . &\three . &\three 1 &\three . &\three .
  &\three . &\three   -8 \\
  . &\three . &\three . &\three . &\three . &\three . &\three . &\three . &\three 1 &\three .
  &\three . &\three  -\frac32 \\
  . &\three . &\three . &\three . &\three . &\three . &\three . &\three . &\three . &\three 1
  &\three . &\three 2 \\
  . &\three . &\three . &\three . &\three . &\three . &\three . &\three . &\three . &\three .
  &\three 1 &\three 2
  \end{array} \right]
  \\
  &\qquad
  \xrightarrow{\mathrm{CIB}}
  \left[ \begin{array}{rrrrrrrrrrrr}
  -448 &\three 128 &\three -32 &\three -24 &\three 16 &\three -24 &\three 2 &\three 16
  &\three 3 &\three -4 &\three -4 &\three 2
  \end{array} \right]
  \end{align*}
\end{example}

\begin{definition} \label{weightvectorbasis}
We construct the \textbf{weight vector basis} of $\Lambda^4 V(n)$ as follows.
We first determine the highest weight vectors for each summand of $\Lambda^4
V(n)$ as described in Definition \ref{definitionEmatrix} and Examples
\ref{highestweightvectors6} and \ref{highestweightvectors8}. For each highest
weight vector $X$ (of weight $w$, say) we apply $F \in \sltwo$ repeatedly $w$
times to obtain weight vectors of weights $w{-}2, w{-}4, \hdots, -w$ forming a
basis of the summand isomorphic to $V(w)$:
  \[
  X, \quad F.X, \quad \frac{1}{2!} F^2.X, \quad \hdots, \quad \frac{1}{w!} F^w.X.
  \]
The set of all these weight vectors is the weight vector basis of $\Lambda^4
V(n)$. The \textbf{standard order} on this basis is as follows: We order the
weight vectors first by decreasing weight of the corresponding highest weight
vector and then by increasing power of $F$ within each summand. (When there is
more than one highest weight vector with the same weight, we order them as in
the canonical integral basis.)
\end{definition}

\begin{example}
The weight vector basis of $\Lambda^4 V(6)$ is displayed in Table
\ref{weightvectorbasistable6}. The weight vector basis of $\Lambda^4 V(8)$ is
displayed in Tables
\ref{weightvectorbasistable8part1}--\ref{weightvectorbasistable8part3}.
\end{example}

\begin{definition} \label{weightvectormatrix}
The \textbf{weight vector matrix} $C$ is the $\binom{n{+}1}{4} \times
\binom{n{+}1}{4}$ matrix which expresses the weight vector basis in terms of
the tensor basis: the $(i,j)$ entry is the coefficient of the $i$-th quadruple
in the $j$-th element of the weight vector basis.
\end{definition}

\begin{definition}
The \textbf{alternating quaternary algebra structure} on $V(n)$ is defined in
terms of structure constants as follows. The inverse $C^{-1}$ of the weight
vector matrix expresses the tensor basis in terms of the weight vector basis.
Let $[p,q,r,s]$ be the $j$-th quadruple in the tensor basis. Column $j$ of
$C^{-1}$ expresses $[p,q,r,s]$ as a linear combination of the elements of the
weight vector basis. Suppose that the highest weight vector for the summand of
$\Lambda^4 V(n)$ isomorphic to $V(n)$ is the $k$-th element of the weight
vector basis.  The entries of $C^{-1}$ in column $j$ and rows $i = k, \hdots,
k{+}n$ are the coefficients of the projection of $[p,q,r,s]$ onto the summand
isomorphic to $V(n)$.  Let $P \colon \Lambda^4 V(n) \to V(n)$ be this
surjective homomorphism of $\sltwo$-modules. The quadruple $[p,q,r,s]$ has
weight $p{+}q{+}r{+}s$, and the summand isomorphic to $V(n)$ has (at most) one
basis vector of this weight. Hence there is at most one nonzero entry in
$C^{-1}$ in column $j$ and rows $i = k, \hdots, k{+}n$. If all these entries
are zero then $P( v_p \wedge v_q \wedge v_r \wedge v_s ) = 0$. If there is a
nonzero entry, say in row $\ell$, then $P( v_p \wedge v_q \wedge v_r \wedge v_s
) = (C^{-1})_{i\ell} \, v_{p+q+r+s}$. The resulting alternating quaternary
algebra structure on $V(n)$ is denoted by $[v_p,v_q,v_r,v_s]$ and defined by
$[v_p,v_q,v_r,v_s] = P( v_p \wedge v_q \wedge v_r \wedge v_s )$.
\end{definition}

\begin{example}
For $n = 6$, from rows 23 to 29 of the matrix inverse we obtain the structure
constants for the alternating quaternary algebra structure on $V(6)$. We ignore
the equations for which $|p{+}q{+}r{+}s| > 6$ since in these cases the result
is obviously zero: there is no vector of the given weight in $V(6)$. The LCM of
the denominators of the coefficients is 120, so we can scale the basis vectors
of $V(n)$ by setting $v'_t = v_t / \root 3 \of {120}$ to obtain integral
structure constants. We use the shorthand notation $[p,q,r,s] = c$ for the
equation $[v'_p,v'_q,v'_r,v'_s] = c v'_{p+q+r+s}$. After making these changes
we obtain the structure constants in Table \ref{structureconstants6}.
\end{example}

\begin{remark}
In this paper we also study the alternating quaternary algebra structures on
$V(10)$ but in this case the tables of the tensor basis, the weight vector
basis, and the structure constants are too large to include.
\end{remark}


\section{Polynomial identities and computational methods}

\begin{definition}
The nonassociative polynomial $I(x_1,\dots, x_n)$ is a \textbf{polynomial
identity} for the algebra $A$ if $I( x_1, \hdots, x_n ) = 0$ for all
$x_1,\dots, x_n\in A$.
\end{definition}

We are concerned with multilinear polynomial identities of degree $n$ for an
alternating quaternary algebra.  This means that each term consists of a
coefficient and a monomial, where the monomial is some permutation of $n$
distinct variables $x_1$, $x_2$, $\hdots$, $x_n$ together with some association
type, by which we mean some placement of brackets representing the quaternary
operation. For any $k$-ary operation, the degree of a monomial has the form $n
= 1 + \ell(k{-}1)$ where $\ell$ is the number of occurrences of the operation
in the monomial. Thus for a quaternary operation the degree of a monomial is
congruent to 1 modulo 3.

In degree 4, we have only the single association type $[-,-,-,-]$; the
alternating property implies that we have only the single monomial
$[x_1,x_2,x_3,x_4]$. In degree 7, the alternating property implies that we have
only one association type $[[-,-,-,-],-,-,-]$ and only $\binom74 = 35$ distinct
multilinear monomials,
  \[
  [
  [x_{\sigma(1)}, x_{\sigma(2)}, x_{\sigma(3)}, x_{\sigma(4)} ],
  x_{\sigma(5)}, x_{\sigma(6)}, x_{\sigma(7)}
  ],
  \]
where $\sigma \in S_7$ is a $(4,3)$-shuffle; that is, $1 \le \sigma(1) <
\sigma(2) < \sigma(3) < \sigma(4) \le 7$ and $1 \le \sigma(5) < \sigma(6) <
\sigma(7) \le 7$. In degree 10, the alternating property implies that we have
two association types,
  \[
  [[[-,-,-,-],-,-,-],-,-,-],
  \qquad
  [[-,-,-,-],[-,-,-,-],-,-],
  \]
and that the corresponding numbers of distinct multilinear monomials are
  \[
  \binom{10}{4,3,3} + \frac12 \binom{10}{4,4,2}
  =
  4200 + 1575
  =
  5775.
  \]
The number $T(\ell)$ of association types which involve $\ell$ occurrences of
an alternating quaternary product is equal to the number of rooted trees with
$\ell$ internal vertices in which each internal vertex has four children; see
Sloane \cite{Sloane}, sequence A036718. The first terms in this sequence are
  \[
  \begin{array}{lrrrrrrrrrrrrrrr}
  \ell & 0 & 1 & 2 & 3 & 4 & 5 & 6 & 7 & 8 & 9 & 10 & 11 & 12 & 13 \\
  T(\ell) & 1 & 1 & 1 & 2 & 4 & 9 & 19 & 45 & 106 & 260 & 643 & 1624 & 4138 & 10683
  \end{array}
  \]
The monomials $[[[a,b,c,d],e,f,g],h,i,j]$ and $[[a,b,c,d],[e,f,g,h],i,j]$
correspond respectively to the following trees:
  \begin{center}
  \Tree [ . [ . [ . $a$ $b$ $c$ $d$ ] $e$ $f$ $g$ ] $h$ $i$ $j$ ]
  \\
  \Tree [ . [ . $a$ $b$ $c$ $d$ ] [ . $e$ $f$ $g$ $h$ ] $i$ $j$ ]
  \end{center}

\subsection{Fill-and-reduce algorithm}

Suppose we wish to find all the multilinear polynomial identities of degree $n$
satisfied by an algebra $A$ of dimension $d$.  We assume that we have chosen a
basis of $A$ and that we know the structure constants with respect to this
basis. We write $m$ for the number of distinct multilinear monomials of degree
$n$, and we assume that these monomials are ordered in some way. We create a
matrix $X$ of size $(m+d) \times m$ and initialize it to zero; the columns of
$M$ correspond bijectively to the monomials. We choose two parameters $p$ and
$s$: we generate pseudorandom integers in the range 0 to $p{-}1$.

We perform the following ``fill-and-reduce'' algorithm until the rank of the
matrix $X$ has remained stable for $s$ iterations:
  \begin{enumerate}
  \item Generate $n$ pseudorandom elements $a_1, \dots, a_n$ of $A$:
      vectors of length $d$ in which the components are integers in the
      range 0 to $p{-}1$.
  \item For $j$ from 1 to $m$ do:
    \begin{enumerate}
    \item Evaluate monomial $j$ by setting the variable $x_k$ equal to
        the vector $a_k$ for $k = 1, \dots, n$ and using the structure
        constants for $A$, obtaining another vector $b$ of length $d$.
    \item Store $b$ as a column vector in column $j$ of $X$ in rows
        $m{+}1$ to $m{+}d$.
    \end{enumerate}
  \item Compute the row canonical form of $X$; the last $d$ rows are now
      zero.
  \end{enumerate}
After this process has terminated, if the nullspace of $X$ is not zero then it
contains candidates for polynomial identities satisfied by $A$. We usually find
that $s = 10$ is a sufficient number of iterations after the rank has
stabilized, but we use $s = 100$ to increase our confidence in the results. We
now compute the canonical integral basis of the nullspace.

\subsection{Module generators algorithm} We assume that we have the canonical
integral basis of the nullspace of the matrix $X$ used in the fill-and-reduce
algorithm. Let $r$ be the number of these basis vectors; these are the
coefficient vectors of polynomial identities satisfied by the algebra $A$.
These identities are linearly independent over $\mathbb{Q}$ but they are not
necessarily independent as generators of the $S_n$-module of identities. We
want to find a minimal set of module generators. We start by sorting the basis
vectors by increasing Euclidean norm. We create a new matrix $M$ of size
$(m{+}n!) \times m$ and initialize it to zero.

We perform the following algorithm for $k$ from 1 to $r$:
  \begin{enumerate}
  \item For $i$ from 1 to $n!$ apply permutation $i$ of the variables $\{
      x_1, \dots, x_n \}$ to basis identity $k$ and store the result in row
      $m{+}i$ of $M$. More precisely, for each nonzero coefficient $c$ of
      the identity, apply permutation $i$ to the corresponding monomial,
      use the alternating property to straighten the monomial, obtain a
      standard basis monomial (with index $j$ say) and store $\pm c$ in
      position $(m+i,j)$ of $M$ (straightening may introduce a sign
      change).
  \item Compute the row canonical form of $M$. If the rank has increased
      from the previous iteration, then we record basis identity $k$ as a
      new generator.
  \end{enumerate}


\section{Multiplicity 1: representation $V(4)$}

In this section and the next we describe computer searches for polynomial
identities satisfied by the two irreducible representations of $\sltwo$ which
admit an alternating quaternary structure which is unique up to a scalar
multiple; we determine all their identities of degree $7$, and show that there
are no new identities in degree 10. For all our calculations we use the
computer algebra system \texttt{Maple}, especially the packages
\texttt{LinearAlgebra} and \texttt{LinearAlgebra[Modular]}.

\begin{theorem} \label{TheoremV(4)}
The vector space of multilinear polynomial identities in degree $7$ for the
alternating quaternary structure on $V(4)$ has dimension 21.
\end{theorem}

\begin{proof}
We use the fill-and-reduce algorithm with $n = 7$, $d = 5$, $m = 35$, $p = 10$
and $s = 100$. We create a matrix $X$ of size $40\times 35$ consisting of an
upper block of size $35 \times 35$ and a lower block of size $5 \times 35$; the
columns are labeled by the ordered basis of multilinear monomials in degree $7$
for an alternating quaternary operation. We generate seven random elements of
$V(4)$ and evaluate the $35$ monomials on these seven elements. We put the $35$
resulting elements of $V(4)$ as column vectors into the lower block of the
matrix. Each of the last five rows of the matrix now contains a linear relation
that must be satisfied by the coefficients of any identity for the alternating
quaternary structure on $V(4)$. We repeat the fill-and-reduce process until the
rank of the matrix stabilizes. The rank reached 14 and did not increase
further. Therefore the nullspace of the matrix has dimension 21.
\end{proof}

\begin{theorem}
Every multilinear polynomial identity in degree 7 for the alternating
quaternary structure on $V(4)$ is a consequence of the alternating property in
degree 4 together with the quaternary derivation identity in degree 7:
  \begin{align*}
  &
  [a,b,c,[d,e,f,g]]
  =
  \\
  &
  [[a,b,c,d],e,f,g] + [d,[a,b,c,e],f,g] + [d,e,[a,b,c,f],g] + [d,e,f,[a,b,c,g]]
  \end{align*}
\end{theorem}

\begin{proof}
We use the module generators algorithm, slightly modified to use less memory.
We create a matrix of size $59 \times 35$ with an upper block of size $35
\times 35$ and a lower block of size $24 \times 35$. We generate all 5040
permutations of seven letters and divide them into 210 groups of 24
permutations. For each of the 21 basis identities, we perform the following
computation. For each group of permutations, we apply the corresponding 24
permutations to the identity, obtain 24 new identities which we store in the
lower block of the matrix, and compute the row canonical form of the matrix.
After all 210 groups of permutations have been processed, the rank of the
matrix is equal to the dimension of the module generated by all the identities
up to and including the current identity. After the first identity has been
processed, the rank of the matrix is 21, which is the same as the entire
nullspace; the rank does not increase as we process the remaining identities.
Therefore every identity is a consequence of the first identity, which has the
form
  \begin{align*}
  &
  [[a,b,c,d],e,f,g] - [[a,b,c,e],d,f,g] + [[a,b,c,f],d,e,g] - [[a,b,c,g],d,e,f]
  \\
  &
  + [[d,e,f,g],a,b,c]
  \end{align*}
Applying the alternating property of the quaternary product, we see that this
identity can be written in the stated form.
\end{proof}

\begin{remark}
The alternating property in degree 4 and the quaternary derivation identity
together define the case $n = 4$ of the variety of $n$-Lie algebras introduced
by Filippov \cite{Filippov1}. Thus the isomorphism $\Lambda^4 V(4) \cong V(4)$
makes $V(4)$ into an alternating quaternary algebra isomorphic to one of the
5-dimensional 4-Lie algebras in Filippov's classification of
$(n{+}1)$-dimensional $n$-Lie algebras.
\end{remark}

For $n = 7$ we can use rational arithmetic for these calculations since the
matrix $X$ is not large. We can extend these calculations to $n = 10$, but we
need to use modular arithmetic to save memory, since the matrix $X$ is very
large: it has 5775 columns (the number of alternating quaternary monomials in
degree 10). The fill-and-reduce algorithm stabilizes at rank 660, and so the
nullspace has dimension 5115. We need to determine which of these identities in
degree 10 are consequences of the quaternary derivation identity in degree 7,
which we denote by $D(a,b,c,d,e,f,g)$. Since this polynomial alternates in
$a,b,c$ we only need to consider six consequences in degree 10, using the
variables $\{a,b,c,d,e,f,g,h,i,j\}$:
  \begin{alignat*}{3}
  &
  D([a,h,i,j],b,c,d,e,f,g), &\;
  &
  D(a,b,c,[d,h,i,j],e,f,g), &\;
  &
  D(a,b,c,d,[e,h,i,j],f,g),
  \\
  &
  D(a,b,c,d,e,[f,h,i,j],g), &\;
  &
  D(a,b,c,d,e,f,[g,h,i,j]), &\;
  &
  [D(a,b,c,d,e,f,g),h,i,j].
  \end{alignat*}
We use a modification of the module generators algorithm to determine that
these identities generate a module of dimension 5115. Since this equals the
dimension of the nullspace from the fill-and-reduce algorithm, it follows that
the alternating quaternary structure on $V(4)$ satisfies no new identities in
degree 10; that is, every identity in degree 10 is a consequence of the known
identities in lower degrees. We used $p = 101$ for these calculations; since
the group algebra of $S_n$ is semisimple over any field of characteristic $p
> n$, and we are studying identities of degree $n = 10$, it follows that any
prime $p > 10$ would give the same dimensions.


\section{Multiplicity 1: representation $V(6)$}

\begin{theorem} \label{TheoremV(6)1}
The vector space of multilinear polynomial identities in degree 7 for the
alternating quaternary structure on $V(6)$ has dimension 1.
\end{theorem}

\begin{proof}
We use the fill-and-reduce algorithm with $n = 7$, $d = 7$, $m = 35$, $p = 10$
and $s = 100$. The details of the computations are similar to those described
in the proof of Theorem \ref{TheoremV(4)}. The rank reached 34 and did not
increase further. Therefore the nullspace of the matrix has dimension 1.
\end{proof}

\begin{theorem} \label{TheoremV(6)2}
Every multilinear polynomial identity in degree $7$ for the alternating
quaternary structure on $V(6)$ is a consequence of the alternating property in
degree 4 together with the quaternary alternating sum identity in degree 7:
  \[
  \sum_{\sigma \in S_7}
  \epsilon(\sigma) \,
  [[a^\sigma,b^\sigma,c^\sigma,d^\sigma],e^\sigma,f^\sigma,g^\sigma]
  \]
\end{theorem}

\begin{proof}
Since the nullspace has dimension 1, this is an immediate corollary of Theorem
\ref{TheoremV(6)1}; we do not need to apply the module generators algorithm.
\end{proof}

\begin{remark}
The referee provided the following alternative proof. The quaternary
alternating sum identity is an alternating multilinear function of 7 variables.
Evaluating this function on the 7-dimensional space $V(6)$ gives a map
$\alpha\colon \Lambda^7 V(6) \to V(6)$. But $\Lambda^7 V(6)$ is 1-dimensional
(it is isomorphic to $V(0)$ as an $\sltwo$-module), and $\alpha$ is invariant
under the action of $\sltwo$. Hence the image of $\alpha$ is an
$\sltwo$-submodule which has dimension 0 or 1. Since $V(6)$ is irreducible, it
has no submodule of dimension 1, and so $\alpha$ must be the zero map.
\end{remark}

\begin{remark}
It is shown in Bremner \cite{Bremner1} (Theorems 3 and 4) that the quaternary
alternating sum identity in degree 7 is satisfied by the following multilinear
operation (the alternating quaternary sum) in every totally associative
quadruple system,
  \[
  [x_1,x_2,x_3,x_4]
  =
  \sum_{\pi \in S_4}
  \epsilon(\pi) \,
  x_{\pi(1)} x_{\pi(2)} x_{\pi(3)} x_{\pi(4)},
  \]
and that the quaternary alternating sum identity of Theorem \ref{TheoremV(6)2}
is a consequence of the quaternary derivation identity of Theorem
\ref{TheoremV(4)}.
\end{remark}

We can extend these calculations to $n = 10$ using modular arithmetic. The
fill-and-reduce algorithm stabilizes at rank 1903, and so the nullspace has
dimension 3872. We need to determine which of these identities in degree 10 are
consequences of the quaternary alternating sum identity in degree 7, which we
denote by $S(a,b,c,d,e,f,g)$. Since this polynomial alternates in all 7
variables, we only need to consider two consequences in degree 10:
  \[
  S([a,h,i,j],b,c,d,e,f,g),
  \qquad
  [S(a,b,c,d,e,f,g),h,i,j].
  \]
We can use an elementary argument to find an upper bound on the dimension of
the submodule generated by these two identities. The first identity alternates
in $a,h,i,j$ and also in $b,c,d,e,f,g$; hence any permutation of the first
identity is equal, up to a sign, to one of $\binom{10}{4} = 210$ possibilities.
Similarly, the second identity alternates in $a,b,c,d,e,f,g$ and in $h,i,j$;
hence any permutation of the second identity is equal, up to a sign, to one of
$\binom{10}{7} = 120$ possibilities.  Altogether we see that the submodule
generated by these two identities has dimension at most 330. (In fact our
computations show that this submodule has dimension 329.) Since this is less
than the dimension of the nullspace from the fill-and-reduce algorithm, the
alternating quaternary structure on $V(6)$ satisfies new identities in degree
10 that are not consequences of the known identities in lower degrees. It is an
open problem to determine generators for the $S_{10}$-module of new identities
in degree 10.


\section{Multiplicity 2: representation $V(8)$}

In this section and the next we describe computer searches for polynomial
identities satisfied by the two irreducible representations of $\sltwo$ which
admit a two-dimensional space of alternating quaternary structures; we
determine all their identities of degree $7$.

Any $\sltwo$-module homomorphism $\Lambda^4 V(8) \to V(8)$ is a linear
combination of the structures $f$ and $g$ from Tables
\ref{structureconstants8part1} and \ref{structureconstants8part2}. Up to a
scalar multiple, we need to consider only the single structure $g$ and the
one-parameter family of structures $f + xg$ for $x \in \mathbb{C}$. For $g$ our
methods are similar to those used for $V(4)$ and $V(6)$. For $f + xg$ we need
to use the Smith normal form to determine the values of the parameter $x$ which
produce a nonzero nullspace. For this we use the Maple command
\texttt{linalg[smith]} instead of \texttt{LinearAlgebra[SmithForm]} since the
former is much more efficient than the latter.

\begin{theorem} \label{TheoremV(8)g}
The vector space of multilinear polynomial identities in degree 7 for the
alternating quaternary structure $g$ on $V(8)$ has dimension 1 and is spanned
by the quaternary alternating sum identity.
\end{theorem}

\begin{proof}
Similar to the proofs of Theorems \ref{TheoremV(6)1} and \ref{TheoremV(6)2}.
\end{proof}

\begin{theorem} \label{TheoremV(8)f+xg}
For any $x \in \mathbb{C}$, the vector space of multilinear polynomial
identities in degree 7 for the alternating quaternary structure $f + xg$ on
$V(8)$ has dimension 1 and is spanned by the quaternary alternating sum
identity.
\end{theorem}

\begin{proof}
In order to determine how the space of identities depends on the parameter $x$,
we use the Smith normal form of a matrix over the polynomial algebra
$\mathbb{C}[x]$. Since the computation of the Smith form performs not only row
operations but also column operations, we must fill the matrix using a suitable
number of trials, and then compute the Smith form once. In the general case, we
create a matrix of size $t(n{+}1) \times m$ where $n$ is the highest weight
(recall that $V(n)$ has dimension $n{+}1$) and $m$ is the number of multilinear
monomials in degree $d$; the matrix consists of $t$ blocks of size $(n{+}1)
\times m$. We choose $t$ so that $t(n{+}1) > m$ in order to guarantee that we
have enough nonzero rows in the matrix to eliminate false nullspace vectors. We
perform the following algorithm:
  \begin{enumerate}
  \item For $b$ from 1 to $t$ do:
    \begin{enumerate}
    \item Generate $d$ pseudorandom vectors of length $n{+}1$
        representing elements of $V(n)$.
    \item For $j$ from 1 to $m$ do:
      \begin{enumerate}
      \item Evaluate the $j$-th alternating quaternary monomial on
          the $d$ pseudorandom vectors to obtain another vector of
          length $n{+}1$ with components which are polynomials in
          the parameter $x$.
      \item Put the resulting vector into column $j$ of block $t$.
      \end{enumerate}
    \end{enumerate}
  \item Compute the Smith normal form of the matrix.
  \end{enumerate}
For $n = 8$ and $d = 7$ we have $m = 35$ and we choose $t = 4$. The entries of
the resulting $36 \times 35$ matrix are quadratic polynomials in the parameter
$x$ since each monomial involves two occurrences of the quaternary operation.
In the Smith normal form of the matrix, the diagonal entries are 1 (34 times)
and 0 (once). It follows that the matrix has a one-dimensional nullspace for
every value of $x$. In Bremner \cite{Bremner1} (Proposition 3, page 85) it is
shown that there a unique 1-dimensional $S_7$-submodule of the 35-dimensional
module with basis consisting of the alternating quaternary monomials in degree
7, and this submodule is spanned by the quaternary alternating sum identity.
Hence the nullspace basis does not depend on the value of the parameter $x$,
and this completes the proof. We checked this result independently by
evaluating the quaternary alternating sum identity on pseudorandom vectors for
the product $f + xg$ with indeterminate $x$ and verifying that the result was
zero.
\end{proof}

\begin{remark}
It is an open problem to determine whether the alternating quaternary
structures on $V(8)$ are isomorphic for all values of the parameter $x$.
\end{remark}


\section{Multiplicity 2: representation $V(10)$}

As in the previous section, any $\sltwo$-module homomorphism $\Lambda^4 V(10)
\to V(10)$ is a linear combination of two structures $f$ and $g$, and we
consider separately the single structure $g$ and the one-parameter family of
structures $f + xg$ for $x \in \mathbb{C}$.

\begin{theorem} \label{TheoremV(10)g}
The vector space of multilinear polynomial identities in degree 7 for the
alternating quaternary structure $g$ on $V(10)$ has dimension 0: every identity
is a consequence of the alternating properties in degree 4.
\end{theorem}

\begin{proof}
Similar to the proofs of Theorem \ref{TheoremV(8)g} except that the matrix
achieves the full rank of 35.
\end{proof}

\begin{theorem} \label{TheoremV(10)f+xg}
For $x = \frac54$, the vector space of multilinear polynomial identities in
degree 7 for the alternating quaternary structure $f + xg$ on $V(10)$ has
dimension 1 and is spanned by the quaternary alternating sum identity. For all
other $x \in \mathbb{C}$, the vector space of multilinear polynomial identities
in degree 7 for the alternating quaternary structure $f + xg$ on $V(10)$ has
dimension 0.
\end{theorem}

\begin{proof}
Similar to the proof of Theorem \ref{TheoremV(8)f+xg} except that now $n = 10$.
As before, the entries of the resulting $44 \times 35$ matrix are quadratic
polynomials in the parameter $x$. In the Smith normal form of this matrix, the
diagonal entries are 1 (28 times) and $x-\frac54$ (7 times). It follows that
the matrix has zero nullspace except in the case $x = \frac54$. We now
specialize to this value of $x$ and consider the structure $f + \frac54 g$; the
rest of the proof is similar to that of Theorems \ref{TheoremV(6)1} and
\ref{TheoremV(6)2}.
\end{proof}


\section{Proof of multiplicity formula} \label{proofsection}

In this section we prove the multiplicity formula of Theorem
\ref{multiplicitytheorem}.  We reduce the problem to a combinatorial question
and apply the theory of P\'olya enumeration.

\begin{lemma} \label{weightspacedimension} \cite[Lemma 5.2]{BremnerHentzel2}
Let $M = \Lambda^k V(n)$ be the $k$-th exterior power of $V(n)$. If $w \in
\mathbb{Z}$ with $kn \ge w \ge -kn$ and $w \equiv kn \, (\mathrm{mod} \, 2)$
then the dimension of the weight space $M_w$ is the number of sequences
$(w_1,w_2,\ldots,w_k) \in \mathbb{Z}^k$ satisfying
  \[
  n \ge w_1 > w_2 > \ldots > w_k \ge -n;
  \;\;
  w_1 + w_2 + \ldots + w_k = w;
  \;\;
  w_1, \hdots, w_k \equiv n \, (\mathrm{mod} \, 2).
  \]
\end{lemma}

We now specialize to $k = 4$ since we are interested in the fourth exterior
power. To compute the multiplicity of $V(n)$ as a direct summand of $\Lambda^4
V(n)$ using Lemmas \ref{multiplicityformula} and \ref{weightspacedimension}, we
must determine the number of quadruples $(p,q,r,s)$ satisfying
  \begin{equation}\label{conditions}
  n \ge p > q > r > s \ge -n;
  \quad
  p + q + r + s = w;
  \quad
  p, q, r, s \equiv n \, (\mathrm{mod} \, 2)
  \end{equation}
for $w = n$ and $w = n + 2$. Let $n$ be a non-negative integer and let $w$ be a
weight of $\Lambda^4 V(n)$: thus $w$ is an integer satisfying
  \[
  4n \ge w \ge -4n,
  \qquad
  w \equiv 0 \, (\mathrm{mod} \, 2 ).
  \]
For integers $p,q,r,s$ satisfying (\ref{conditions}) we define
  \[
  P' = p + n, \quad
  Q' = q + n, \quad
  R' = r + n, \quad
  S' = s + n.
  \]
Then $(P',Q',R',S')$ is a quadruple of even integers satisfying
  \[
  2n \ge P' > Q' > R' > S' \ge 0,
  \quad
  P' + Q' + R' + S' = W',
  \quad
  W' = w + 4n.
  \]
We need to count the number of partitions of $W'$ into four distinct
nonnegative even parts less than or equal to $2n$. We only need $W' = 5n$ and
$W' = 5n+2$ corresponding to $w = n$ and $w = n+2$. It is clear that if $n$ is
odd then there are no solutions in both cases, so $V(n)$ does not occur as a
summand of $\Lambda^4 V(n)$: the multiplicity is zero. Therefore we may assume
that $n$ is even and define
  \[
  P = \frac{p+n}{2}, \quad
  Q = \frac{q+n}{2}, \quad
  R = \frac{r+n}{2}, \quad
  S = \frac{s+n}{2}, \quad
  W = \frac{w+4n}{2}.
  \]
Then $(P,Q,R,S)$ is a quadruple of integers satisfying
  \[
  n \ge P > Q > R > S \ge 0,
  \qquad
  P + Q + R + S = W.
  \]

\begin{definition} \cite[page 612]{WuChao}
If $G$ is a subgroup of the symmetric group $S_n$ then the \textbf{cycle index}
of $G$ is the following polynomial in the indeterminates $x_1, x_2, \dots,
x_n$:
  \[
  Z_G{\left(x_1,x_2,\hdots,x_n\right)}
  =
  \frac{1}{|G|}\sum_{\sigma \in G}{x^{b_1}_1 x^{b_2}_2 \cdots x^{b_n}_n};
  \]
here $b_i$ is the number of cycles of length $i$ in the disjoint cycle
factorization of $\sigma$.
\end{definition}

\begin{lemma} \cite[page 36]{HararyPalmer} \label{book1}
The cycle index of the alternating group $A_n$ is
  \[
  Z_{A_n} \left( x_1, x_2, \hdots, x_n \right)
  =
  Z_{S_n} \left( x_1, x_2, \hdots, x_n \right)
  +
  Z_{S_n} \left( x_1, -x_2, \hdots, (-1)^{n-1} x_n \right).
  \]
\end{lemma}

\begin{proof}
The definition of cycle index gives
  \[
  Z_{A_n}\left(x_1,x_2,\hdots,x_n\right)
  =
  \frac{2}{n!}
  \bigg[
  \sum_{\sigma \in S_n}
  x^{b_1}_1 x^{b_2}_2 \cdots x^{b_n}_n
  -
  \!\!\!\!\!\!
  \sum_{\sigma\in S_n \backslash A_n}
  x^{b_1}_1 x^{b_2}_2 \cdots x^{b_n}_n
  \bigg].
  \]
Since $\sigma\in S_{n}\backslash A_{n}$ if and only if $\sigma$ has an odd
number of even length cycles, we get
  \allowdisplaybreaks
  \begin{align*}
  &
  Z_{A_n}{\left(x_1,x_2,\hdots,x_n\right)}
  =
  \\
  &
  \frac{2}{n!} \cdot \frac12
  \bigg[
  \sum_{\sigma \in S_n}
  x^{b_1}_1 x^{b_2}_2 \cdots x^{b_n}_n
  +
  \sum_{\sigma \in S_n}
  x_1^{b_1} (-x_2)^{b_2} \cdots \big( (-1)^{n-1} x_n \big)^{b_n}
  \bigg].
  \end{align*}
This completes the proof.
\end{proof}

The next result is the special case $k = 4$ of Theorem 2 in Wu and Chao
\cite{WuChao}; but note that we allow $0 \in S$.

\begin{proposition} \label{cycleindextheorem}
If $S$ is a set of non-negative integers then the number of partitions of an
integer $n$ into four distinct parts in $S$ is the coefficient of $x^n$ in
  \allowdisplaybreaks
  \begin{align*}
  &Z_{A_4}
  \left(
  \sum_{i\in S} x^i,\sum_{i\in S} x^{2i},\sum_{i\in S} x^{3i},\sum_{i\in S} x^{4i}
  \right)
  -
  Z_{S_4}
  \left(
  \sum_{i\in S} x^i,\sum_{i\in S} x^{2i},\sum_{i\in S} x^{3i},\sum_{i\in S} x^{4i}
  \right).
  \end{align*}
\end{proposition}

\begin{corollary} \label{polyacorollary}
If $S$ is a set of non-negative integers then the  number of partitions of a
positive integer $n$ into four distinct parts in $S$ is the coefficient of
$x^n$ in
  \[
  Z_{S_4}
  \left(
  \sum_{i \in S} x^i,
  - \sum_{i \in S} x^{2i},
  \sum_{i \in S} x^{3i},
  - \sum_{i \in S} x^{4i}
  \right).
  \]
\end{corollary}

\begin{proof}
Take $n = 4$ in Lemma \ref{book1}, set $x_j = \sum_{i \in S} x^{ji}$, and apply
Proposition \ref{cycleindextheorem}.
\end{proof}

\begin{definition} \label{Pdefinition}
For us $S = \{0,1,\hdots,n\}$ so we define the following polynomials:
  \[
  P_n(x)
  =
  Z_{S_4}
  \left(
  \sum_{i=0}^n x^i, - \sum_{i=0}^n x^{2i}, \sum_{i=0}^n x^{3i}, - \sum_{i=0}^n x^{4i}
  \right).
  \]
\end{definition}

\begin{lemma} \label{coefficientformula}
We have
  \[
  \left( \sum_{i=0}^n x^i \right)^t
  =
  \sum_{\ell=0}^{nt}
  \left[
  \sum_{k=0}^{\min \left( t, \left\lfloor \frac{\ell}{n+1} \right\rfloor \right)}
  (-1)^k
  \binom{t}{k}
  \binom{\ell{-}(n{+}1)k{+}t{-}1}{t{-}1}
  \right]
  x^\ell.
  \]
\end{lemma}

\begin{proof}
We use these three familiar identities:
  \[
  \frac{1{-}x^{n+1}}{1{-}x}
  =
  \sum_{i=0}^n x^i,
  \;
  \left( 1{-}x^{n+1} \right)^t
  =
  \sum_{k=0}^{t}
  (-1)^k
  \binom{t}{k}
  x^{(n+1)k},
  \;
  \frac{1}{(1{-}x)^t}
  =
  \sum_{j=0}^\infty
  \binom{j{+}t{-}1}{t{-}1}
  x^j.
  \]
We obtain
  \begin{equation} \label{combinedformula}
  \left(
  \sum_{i=0}^n x^i
  \right)^t
  =
  \frac{(1-x^{n+1})^t}{(1-x)^t}
  =
  \sum_{k=0}^t
  \sum_{j=0}^\infty
  (-1)^k
  \binom{t}{k}
  \binom{j{+}t{-}1}{t{-}1}
  x^{(n+1)k+j}.
  \end{equation}
We set $\ell = (n{+}1)k+j$ so that $\ell{-}j = (n{+}1)k$ and $\ell{-}j \equiv 0
\; (\mathrm{mod} \; n{+}1)$. We also have $k = (\ell{-}j)/(n{+}1)$ and so $k
\le \lfloor \ell/(n{+}1) \rfloor$. Substituting $j = \ell - (n{+}1)k$ in
\eqref{combinedformula}, and noting that $nt$ is the largest power of $x$, we
obtain the stated formula.
\end{proof}

\begin{definition}
We use the following notation:
  \[
  \Delta_m^n
   =
  \begin{cases}
  1 &\text{if $n \equiv 0 \, (\mathrm{mod} \, m)$}  \\
  0 &\text{otherwise}
  \end{cases},
  \qquad
  \Delta_{s,m}^n
  =
  \begin{cases}
  1 &\text{if $n \equiv s \, (\mathrm{mod} \, m)$} \\
  0 &\text{otherwise}
  \end{cases}.
\]
\end{definition}

\begin{definition} \label{alpha}
We consider the following integer-valued functions of $n$:
  \[
  \alpha(n) = \left\lceil \frac{n}{4} \right\rceil,
  \quad
  \beta(n) = \left\lceil \frac{3n}{4} \right\rceil,
  \quad
  \gamma(n) = \left\lfloor \frac{3n{-}2}{4} \right\rfloor,
  \quad
  \delta(n) = \left\lfloor \frac{5n}{6} \right\rfloor.
  \]
\end{definition}

\begin{proposition} \label{dimension5n}
For even $n \in \mathbb{Z}$, the number of solutions $P,Q,R,S \in \mathbb{Z}$
to
  \[
  n \ge P > Q > R > S \ge 0,
  \qquad
  P + Q + R + S = \frac{5n}{2},
  \]
equals
  \allowdisplaybreaks
  \begin{align*}
  &
  \frac{23}{1152} \, n^3
  -
  \frac{29}{96} \, n^2
  +
  \frac{1}{288}
  \Big( -36 \alpha(n) + 180 \beta(n) + 36 \gamma(n) + 27 \Delta_4^n - 167 \Big) n
  \\
  &
  +
  \frac{1}{24}
  \Big(
  6 \alpha(n)^2 - 6 \beta(n)^2 - 6 \gamma(n)^2
  + 12 \beta(n) - 12 \gamma(n) + 8 \delta(n)
  + 3 \Delta_4^n - 6 \Delta_8^n - 3
  \Big).
  \end{align*}
\end{proposition}

\begin{proof}
By Corollary \ref{polyacorollary}, we need to find the coefficient of
$x^{5n/2}$ in the polynomial $P_n(x)$ of Definition \ref{Pdefinition}. The
cycle index of $S_4$ is
  \[
  Z_{S_4}(x_1,x_2,x_3,x_4)
  =
  \frac{1}{24}
  \big(
  x_1^4 + 6 x_1^2 x_2 + 8 x_1 x_3 + 3x^2_2 + 6 x_4
  \big).
  \]
For the first four terms, we need to evaluate the following products:
  \[
  A = \bigg( \sum_{i=0}^n x^i \bigg)^4,
  \;
  B = \bigg( \sum_{i=0}^n x^i \bigg)^2 \sum_{i=0}^n x^{2i},
  \;
  C = \sum_{i=0}^n x^i \sum_{i=0}^n x^{3i},
  \;
  D = \bigg( \sum_{i=0}^n x^{2i} \bigg)^2.
  \]
Lemma \ref{coefficientformula} gives
  \[
  A
  =
  \sum_{\ell=0}^{4n}
  \sum_{k=0}^{\min\left( 4, \left\lfloor \frac{\ell}{n+1} \right\rfloor \right)}
  (-1)^k
  \binom{4}{k}
  \binom{\ell{-}(n{+}1)k{+}3}{3}
  x^\ell.
  \]
Similarly,
  \[
  B
  =
  \Bigg(
  \sum_{\ell=0}^{2n}
  \sum_{k=0}^{\min\left( 2, \left\lfloor \frac{\ell}{n+1} \right\rfloor \right)}
  (-1)^k
  \binom{2}{k}
  \big( \ell{-}(n{+}1)k{+}1 \big)
  x^\ell
  \Bigg)
  \sum_{i=0}^n x^{2i}.
  \]
The upper limit of $k$ is 0 for $0 \le \ell \le n$, and 1 for $n{+}1 \le \ell
\le 2n$. Hence
  \allowdisplaybreaks
  \begin{align*}
  B
  &=
  \Bigg(
  \sum_{\ell=0}^n
  \big( \ell{+}1 \big)
  x^\ell
  +
  \sum_{\ell=n+1}^{2n}
  \big[ (\ell{+}1){-}2(\ell{-}n) \big]
  x^\ell
  \Bigg)
  \sum_{i=0}^n x^{2i}
  \\
  &=
  \Bigg(
  \sum_{\ell=0}^n
  \big( \ell{+}1 \big)
  x^\ell
  +
  \sum_{\ell=n+1}^{2n}
  \big( 2n{-}\ell{+}1 \big)
  x^\ell
  \Bigg)
  \sum_{i=0}^n x^{2i}
  \\
  &=
  \sum_{\ell=0}^n
  \sum_{m=0}^n
  (\ell{+}1)
  x^{\ell+2m}
  +
  \sum_{\ell=n+1}^{2n}
  \sum_{m=0}^n
  (2n{-}\ell{+}1)
  x^{\ell+2m}.
  \end{align*}
We now set $p = \ell {+} 2m$, so that $\ell = p{-}2m$.  For $0 \le \ell \le n$
we have $0 \le p{-}2m \le n$ and so $\frac12 (p{-}n) \le m \le \frac12 p$, but
$m \in \mathbb{Z}$ so $\lceil \frac12 (p{-}n) \rceil \le m \le \lfloor \frac12
p \rfloor$; since also $0 \le m \le n$ we get $\max( 0, \lceil \frac12 (p{-}n)
\rceil ) \le m \le \min( n, \lfloor \frac12 p \rfloor )$. Similarly, for $n{+}1
\le \ell \le 2n$ we obtain $\max \left( 0, \left\lceil \frac12(p{-}2n)
\right\rceil \right) \le m \le \min \left( n, \left\lfloor \frac12
(p{-}(n{+}1)) \right\rfloor \right)$. Therefore
  \[
  B
  =
  \sum_{p=0}^{3n}
  \sum_{m=\max\left(0,\left\lceil\frac{p-n}{2}\right\rceil\right)}
  ^{\min\left(n,\left\lfloor\frac{p}{2}\right\rfloor\right)}
  \!\!\!\!\!\!\!\!
  \left(
  p{-}2m {+} 1
  \right)
  x^p
  +
  \sum_{p=n+1}^{4n}
  \sum_{m=\max\left(0,\left\lceil\frac{p-2n}{2}\right\rceil\right)}
  ^{\min\left(n,\left\lfloor\frac{p-(n+1)}{2}\right\rfloor\right)}
  \!\!\!\!\!\!\!\!
  \left(
  2n {-} (p{-}2m) {+} 1
  \right)
  x^p.
  \]
Using a similar change of index we obtain
  \allowdisplaybreaks
  \begin{align*}
  C
  &=
  \sum_{i=0}^n
  \sum_{j=0}^n
  x^{i+3j}
  =
  \sum_{p=0}^{4n}
  \sum_{m=\max(0,\lceil\frac{p-n}{3}\rceil)}^{\min(n,\lfloor\frac{p}{3}\rfloor)}
  \!\!\!\!\!\!
  x^p
  \\
  &=
  \sum_{p=0}^{4n}
  \left[
  \min\left(n,\left\lfloor\frac{p}{3}\right\rfloor\right){-}
  \max\left(0,\left\lceil\frac{p{-}n}{3}\right\rceil\right){+}1
  \right]
  x^p.
  \end{align*}
Replacing $x$ by $x^2$ in Lemma \ref{coefficientformula} gives
  \[
  D
  =
  \sum_{\ell=0}^n
  (\ell{+}1)
  x^{2\ell}
  +
  \sum_{\ell=n+1}^{2n}
  (2n{-}\ell{+}1)
  x^{2\ell}.
  \]
We now write
  \[
  E = \sum_{\ell=0}^n x^{4\ell},
  \]
and obtain
  \allowdisplaybreaks
  \begin{align*}
  &
  A - 6 B + 8 C + 3 D - 6 E
  \\
  &=
  \sum_{p=0}^{4n}
  \sum_{k=0}^{\min\left( 4, \left\lfloor \frac{p}{n+1} \right\rfloor \right)}
  (-1)^k
  \binom{4}{k}
  \binom{p{-}(n{+}1)k{+}3}{3}
  x^p
  \\
  &\quad
  - 6
  \Bigg[
  \sum_{p=0}^{3n}
  \sum_{m=\max\left(0,\left\lceil\frac{p-n}{2}\right\rceil\right)}
  ^{\min\left(n,\left\lfloor\frac{p}{2}\right\rfloor\right)}
  \!\!\!\!\!\!\!\!
  \left(
  p{-}2m {+} 1
  \right)
  x^p
  +
  \sum_{p=n+1}^{4n}
  \sum_{m=\max\left(0,\left\lceil\frac{p-2n}{2}\right\rceil\right)}
  ^{\min\left(n,\left\lfloor\frac{p-(n+1)}{2}\right\rfloor\right)}
  \!\!\!\!\!\!\!\!
  \left(
  2n {-} (p{-}2m) {+} 1
  \right)
  x^p
  \Bigg]
  \\
  &\quad
  +
  8
  \Bigg[
  \sum_{p=0}^{4n}
  \left[
  \min\left(n,\left\lfloor\frac{p}{3}\right\rfloor\right){-}
  \max\left(0,\left\lceil\frac{p{-}n}{3}\right\rceil\right){+}1
  \right]
  x^p
  \Bigg]
  \\
  &\quad
  +
  3
  \Bigg[
  \sum_{\ell=0}^n
  (\ell{+}1)
  x^{2\ell}
  +
  \sum_{\ell=n+1}^{2n}
  (2n{-}\ell{+}1)
  x^{2\ell}
  \Bigg]
  -
  6
  \sum_{\ell=0}^n x^{4\ell}.
  \end{align*}
We need the coefficient $T$ of $x^{5n/2}$ in the last equation:
  \allowdisplaybreaks
  \begin{align*}
  T
  &=
  \sum_{k=0}^{\left\lfloor\frac{5n}{2(n+1)}\right\rfloor}
  (-1)^k
  \binom{4}{k}
  \binom{\frac{5n}{2}{-}(n{+}1)k{+}3}{3}
  \\
  &\quad
  -
  6
  \Bigg[
  \sum_{m=\lceil\frac{3n}{4}\rceil}^n
  \left(\frac{5n}{2}{-}2m{+}1\right)
  +
  \sum_{m=\left\lceil\frac{n}{4}\right\rceil}^{\left\lfloor\frac{3n-2}{4}\right\rfloor}
  \left(2m{-}\frac{n}{2}{+}1\right)
  \Bigg]
  \\
  &\quad
  +
  8
  \left( \left\lfloor \frac{5n}{6} \right\rfloor {-} \frac{n}{2} {+} 1 \right)
  +
  3 \,
  \delta_4^n \left( 0 {+} \frac{3n}{4} {+} 1 \right)
  -
  6 \,
  \delta_8^n.
  \end{align*}
For $n = 0$ and $n = 2$ we get $T = 0$; this is expected since the
$\sltwo$-modules $V(0)$ and $V(2)$ have dimensions 1 and 3 respectively, so in
both cases $\Lambda^4 V(n)$ is $\{0\}$.  For $n \ge 4$ the upper limit of $k$
is 2, and we use the formula
  \[
  \sum_{m=a}^b m = \frac12 (b-a+1)(b+a).
  \]
We obtain
  \allowdisplaybreaks
  \begin{align*}
  T
  &=
  \binom{\frac{5n}{2}{+}3}{3}
  -
  4
  \binom{\frac{3n}{2}{+}2}{3}
  +
  6
  \binom{\frac{n}{2}{+}1}{3}
  -
  6
  \left(n{-}\left\lceil\frac{3n}{4}\right\rceil{+}1\right)
  \left(\frac{5n}{2}{+}1\right)
  \\
  &\quad
  +
  6
  \left(n{-}\left\lceil\frac{3n}{4}\right\rceil{+}1\right)
  \left(n{+}\left\lceil\frac{3n}{4}\right\rceil\right)
  -
  6
  \left(\left\lfloor\frac{3n{-}2}{4}\right\rfloor{-}\left\lceil\frac{n}{4}\right\rceil{+}1\right)
  \left(-\frac{n}{2}{+}1\right)
  \\
  &\quad
  -
  6
  \left(\left\lfloor\frac{3n{-}2}{4}\right\rfloor{-}\left\lceil\frac{n}{4}\right\rceil{+}1\right)
  \left(\left\lfloor\frac{3n{-}2}{4}\right\rfloor{+}\left\lceil\frac{n}{4}\right\rceil\right)
  +
  8 \left\lfloor\frac{5n}{6}\right\rfloor - 4 n + 8
  \\
  &\quad
  +
  3 \, \delta_4^n \left( \frac{3n}{4} {+} 1 \right)
  -
  6 \, \delta_8^n.
  \end{align*}
Expanding this and collecting terms with the same power of $n$ gives
  \allowdisplaybreaks
  \begin{align*}
  &
  \frac{23}{48} n^3
  -
  \frac{29}{4} n^2
  +
  \frac{1}{12}
  \Big(
  -36 \alpha(n) + 180 \beta(n) + 36 \gamma(n) + 27 \Delta_4^n - 167
  \Big)
  n
  \\
  &
  +
  \Big(
  6 \alpha(n)^2 - 6 \beta(n)^2 - 6 \gamma(n)^2
  + 12 \beta(n) + 8 \delta(n) - 12 \gamma(n)
  + 3 \Delta_4^n - 6 \Delta_8^n - 3
  \Big).
  \end{align*}
We check that this gives $T = 0$ for $n = 0$ and $n = 2$. Finally, we divide by
24.
\end{proof}

\begin{corollary} \label{dimension5ncorollary}
For even $n \in \mathbb{Z}$, write $n = 24 q + r$ with $q, r \in \mathbb{Z}$
and $0 \le r < 24$.  The dimension $\dim [ \Lambda^4 V(n) ]_n$, of the weight
space of weight $n$ in the $\sltwo$-module $\Lambda^4 V(n)$, is given in Table
\ref{multiplicitytable}. (Note the denominator 1152.)
\end{corollary}

\begin{proof}
The dimension is given by the formula of Proposition \ref{dimension5n}. The LCM
of the denominators of the functions $\alpha(n)$, $\beta(n)$, $\gamma(n)$,
$\delta(n)$ and the periods of the functions $\Delta_4^n$ and $\Delta_8^n$
equals 24. Hence the dimension is given by a cubic polynomial in $n$ which
depends on the remainder of $n$ modulo 24.  Considering each remainder $r$
separately, we obtain the polynomials in the first column of Table
\ref{multiplicitytable}.
\end{proof}

\begin{definition} \label{epsilon}
We consider the following integer-valued functions of $n$:
  \[
  \epsilon(n) = \left\lfloor\frac{3n}{4}\right\rfloor, \quad
  \zeta(n) = \left\lceil\frac{n{+}2}{4}\right\rceil, \quad
  \eta(n) = \left\lceil\frac{3n{+}2}{4}\right\rceil, \quad
  \theta(n) = \left\lfloor\frac{5n{+}2}{6}\right\rfloor.
  \]
\end{definition}

\begin{proposition} \label{dimension5n+2}
For even $n \in \mathbb{Z}$, the number of solutions $P,Q,R,S \in \mathbb{Z}$
to
  \[
  n \ge P > Q > R > S \ge 0,
  \qquad
  P + Q + R + S = \frac{5n{+}2}{2},
  \]
equals
  \allowdisplaybreaks
  \begin{align*}
  &
  \frac{23}{1152} \, n^3
  -
  \frac{21}{64} \, n^2
  +
  \frac{1}{288}
  \Big( 36 \epsilon(n) - 36 \zeta(n) + 180 \eta(n) + 27 \Delta_{4,2}^n - 254 \Big) n
  \\
  &
  +
  \frac{1}{48}
  \Big(
  - 12 \epsilon(n)^2 + 12 \zeta(n)^2 - 12 \eta(n)^2
  - 12 \epsilon(n) - 12 \zeta(n) + 36 \eta(n) + 16 \theta(n)
  \\
  &\qquad\qquad
  + 3 \Delta_{4,2}^n - 12 \Delta_{8,6}^n - 24
  \Big).
  \end{align*}
\end{proposition}

\begin{proof}
Similar to the proof of Proposition \ref{dimension5n}.
\end{proof}

\begin{corollary}
For even $n \in \mathbb{Z}$, write $n = 24 q + r$ with $q, r \in \mathbb{Z}$
and $0 \le r < 24$.  The dimension $\dim [ \Lambda^4 V(n) ]_{n+2}$, of the
weight space of weight $n{+}2$ in the $\sltwo$-module $\Lambda^4 V(n)$, is
given in Table \ref{multiplicitytable}. (Note the denominator 1152.)
\end{corollary}

\begin{proof}
Similar to the proof of Corollary \ref{dimension5ncorollary}.
\end{proof}

\begin{theorem}
The multiplicity $\mathrm{mult}(n) = \dim \mathrm{Hom}_{\sltwo}\big( \Lambda^4
V(n), V(n) \big)$ of the irreducible representation $V(n)$ of $\sltwo$ as a
summand of its fourth exterior power $\Lambda^4 V(n)$ is given in Table
\ref{multiplicitytable}. (Note the denominator 1152.)
\end{theorem}

\begin{proof}
Apply Lemma \ref{multiplicityformula} to Propositions \ref{dimension5n} and
\ref{dimension5n+2}.
\end{proof}

\begin{remark}
We can use Corollary \ref{polyacorollary} to obtain another proof of the
decompositions in Remark \ref{completedecomposition}. For $n = 4, 6, 8, 10$ we
compute the polynomial $P_n(x)$ from Definition \ref{Pdefinition}. In each
case, the coefficient of $x^{(w+4n)/2}$ is $\dim [ \Lambda^4 V(n) ]_w$, and we
then apply Lemma \ref{multiplicityformula} to find the multiplicity of $V(w)$
in $\Lambda^4 V(n)$:
  \begin{align*}
  P_4(x)
  &=
  x^{10} + x^9 + x^8 + x^7 + x^6,
  \\
  P_6(x)
  &=
  x^{18} + x^{17} + 2 x^{16} + 3 x^{15} + 4 x^{14} + 4 x^{13} + 5 x^{12} + 4 x^{11}
  + 4 x^{10} + 3 x^9 + 2 x^8
  \\
  &\quad
  + x^7 + x^6,
  \\
  P_8(x)
  &=
  x^{26} + x^{25} + 2 x^{24} + 3 x^{23} + 5 x^{22} + 6 x^{21} + 8 x^{20} + 9 x^{19}
  + 11 x^{18} + 11 x^{17}
  \\
  &\quad
  + 12 x^{16} + 11 x^{15} + 11 x^{14} + 9 x^{13} + 8 x^{12} + 6 x^{11} + 5 x^{10} + 3 x^9
  + 2 x^8 + x^7
  \\
  &\quad
  + x^6,
  \\
  P_{10}(x)
  &=
  x^{34} + x^{33} + 2 x^{32} + 3 x^{31} + 5 x^{30} + 6 x^{29} + 9 x^{28} + 11 x^{27}
  + 14 x^{26} + 16 x^{25}
  \\
  &\quad
  + 19 x^{24} + 20 x^{23} + 23 x^{22} + 23 x^{21} + 24 x^{20} + 23 x^{19} + 23 x^{18}
  + 20 x^{17}
  \\
  &\quad
  + 19 x^{16} + 16 x^{15} + 14 x^{14} + 11 x^{13} +  9 x^{12} + 6 x^{11} + 5 x^{10}
  + 3 x^9 + 2 x^8
  \\
  &\quad
  + x^7 + x^6.
  \end{align*}
\end{remark}


\section{Conclusion}

We recovered a 5-dimensional 4-Lie algebra from the isomorphism $\Lambda^4 V(4)
\cong V(4)$. This algebra satisfies the quaternary derivation identity $D$, and
hence also the quaternary alternating sum identity $S$. We found that the
identity $S$ is also satisfied by the unique structure on $V(6)$, every
structure $f + xg$ on $V(8)$, and the structure $f + \frac54 g$ on $V(10)$.  By
Bremner \cite{Bremner1} it is known that the quaternary alternating sum
operation in a totally associative quadruple system also satisfies $S$. This
raises the question whether the structures which satisfy $S$ can be embedded
into totally associative quadruple systems if the original associative
operation is replaced by the quaternary alternating sum. An affirmative answer
to this question would provide a partial generalization of the
Poincar\'e-Birkhoff-Witt theorem for Lie algebras; see Pozhidaev
\cite{Pozhidaev} for related work.  Simple associative $n$-tuple systems were
classified by Carlsson \cite{Carlsson}. In particular, for $n = 4$, any simple
associative quadruple system is isomorphic to a subspace of matrices of the
form
  \[
  \left[
  \begin{array}{ccc}
  0 & 0 & Z \\
  X & 0 & 0 \\
  0 & Y & 0
  \end{array}
  \right]
  \]
where $X$, $Y$, $Z$ have sizes $q \times p$, $r \times q$, $p \times r$
(respectively) and $p, q, r$ are positive integers. It would be very
interesting to determine whether any of the alternating quaternary algebras
presented in this paper are isomorphic to such a subspace of matrices under the
quaternary alternating sum.


\section*{Acknowledgements}

We thank the referee for a careful reading of the original version and for
bringing to our attention a number of errors and inconsistencies.  This work
forms part of the doctoral thesis of the second author, written under the
supervision of the first author. The first author was partially supported by
NSERC, the Natural Sciences and Engineering Research Council of Canada. The
second author was supported by a University Graduate Scholarship from the
University of Saskatchewan.




  \begin{table}
  {
  \begin{center}
  \[
  \begin{array}{r|l|l|l}
  r  &
  1152 \, \dim [ \Lambda^4 V(n) ]_n &
  1152 \, \dim [ \Lambda^4 V(n) ]_{n+2} &
  1152 \; \mathrm{mult}[V(n)]
  \\
  \hline
  &&
  \\[-10pt]
  0 & 23 n^3 - 42 n^2 + 48 n  & 23 n^3  - 72 n^2  - 48 n & 30n^2+96n
  \\
  2 & 23 n^3 - 42 n^2 - 60 n + 104  & 23 n^3  - 72 n^2  + 60 n - 16 & 30n^2-120n+120
  \\
  4 & 23 n^3 - 42 n^2 + 48 n + 160  & 23 n^3  - 72 n^2  - 48 n - 128 & 30n^2+96n+288
  \\
  6 & 23 n^3 - 42 n^2 - 60 n + 360  & 23 n^3  - 72 n^2  + 60 n - 432 & 30n^2-120n+792
  \\
  8 & 23 n^3 - 42 n^2 + 48 n - 256  & 23 n^3  - 72 n^2  - 48 n + 128 & 30n^2+96n-384
  \\
  10 & 23 n^3 - 42 n^2 - 60 n + 232 & 23 n^3  - 72 n^2  + 60 n - 272 & 30n^2-120n+504
  \\
  12 & 23 n^3 - 42 n^2 + 48 n + 288 & 23 n^3  - 72 n^2  - 48 n & 30n^2+96n+288
  \\
  14 & 23 n^3 - 42 n^2 - 60 n + 104  & 23 n^3  - 72 n^2  + 60 n - 304 & 30n^2-120n+408
  \\
  16 & 23 n^3 - 42 n^2 + 48 n - 128 & 23 n^3  - 72 n^2  - 48 n - 128 & 30n^2+96n
  \\
  18 & 23 n^3 - 42 n^2 - 60 n + 360 & 23 n^3  - 72 n^2  + 60 n - 144 & 30n^2-120n+504
  \\
  20 & 23 n^3 - 42 n^2 + 48 n + 32 & 23 n^3  - 72 n^2  - 48 n + 128 & 30n^2+96n-96
  \\
  22 & 23 n^3 - 42 n^2 - 60 n + 232 & 23 n^3  - 72 n^2  + 60 n - 560 & 30n^2-120n+792
  \end{array}
  \]
  \end{center}
  \caption{Dimensions of weight spaces $n$ and $n{+}2$ for $\Lambda^4 V(n)$}
  \label{multiplicitytable}
  }
  {
  \begin{center}
  \[
  \begin{array}{r|rrrrrrrrrrrr}
  & r=0 & 2 & 4 & 6 & 8 & 10 & 12 & 14 & 16 & 18 & 20 & 22
  \\ \hline
  &&&&&&&&&&&&
  \\[-10pt]
  q=0
  &\three 0 &\three 0 &\three 1 &\three 1 &\three 2 &\three 2 &\three 5 &\three 4
  &\three 8 &\three 7 &\three 12 &\three 11
  \\
  1 &\three 17 &\three 15 &\three 23 &\three 21 &\three 29 &\three 27 &\three 37
  &\three 34 &\three 45 &\three 42 &\three 54 &\three 51
  \\
  2 &\three 64 &\three 60 &\three 75 &\three 71 &\three 86 &\three 82 &\three 99
  &\three 94 &\three 112 &\three 107 &\three 126 &\three 121
  \\
  3 &\three 141 &\three 135 &\three 157 &\three 151 &\three 173 &\three 167
  &\three 191 &\three 184 &\three 209 &\three 202 &\three 228 &\three 221
  \\
  4 &\three 248 &\three 240 &\three 269 &\three 261 &\three 290 &\three 282
  &\three 313 &\three 304 &\three 336 &\three 327 &\three 360 &\three 351
  \\
  5 &\three 385 &\three 375 &\three 411 &\three 401 &\three 437 &\three 427
  &\three 465 &\three 454 &\three 493 &\three 482 &\three 522 &\three 511
  \\
  6 &\three 552 &\three 540 &\three 583 &\three 571 &\three 614 &\three 602
  &\three 647 &\three 634 &\three 680 &\three 667 &\three 714 &\three 701
  \\
  7 &\three 749 &\three 735 &\three 785 &\three 771 &\three 821 &\three 807
  &\three 859 &\three 844 &\three 897 &\three 882 &\three 936 &\three 921
  \\
  8 &\three 976 &\three 960 &\three 1017 &\three 1001 &\three 1058 &\three 1042
  &\three 1101 &\three 1084 &\three 1144 &\three 1127 &\three 1188 &\three 1171
  \\
  9 &\three 1233 &\three 1215 &\three 1279 &\three 1261 &\three 1325 &\three 1307
  &\three 1373 &\three 1354 &\three 1421 &\three 1402 &\three 1470 &\three 1451
  \end{array}
  \]
  \end{center}
  \caption{Multiplicities $\mathrm{mult}(n)$ for $n = 24q + r$ with $0 \le q \le 9$
  ($n$ even)}
  \label{multiplicityexample}
  }
  {
  \[
  \begin{array}{rrrrrr}
   4 & [ \quad 4,& 2,&  0,& -2 \quad ] \\
   2 & [ \quad 4,& 2,&  0,& -4 \quad ] \\
   0 & [ \quad 4,& 2,& -2,& -4 \quad ] \\
  -2 & [ \quad 4,& 0,& -2,& -4 \quad ] \\
  -4 & [ \quad 2,& 0,& -2,& -4 \quad ]
  \end{array}
  \qquad\qquad
  V(4)
  \left\{
  \begin{array}{rl}
   4\colon &  [1] \\
   2\colon &  [4] \\
   0\colon &  [6] \\
  -2\colon &  [4] \\
  -4\colon &  [1]
  \end{array}
  \right.
  \]
  \caption{Tensor basis and weight vector basis of $\Lambda^4 V(4)$}
  \label{tensorweightvectorbasistable4}
  }
  {
  \begin{alignat*}{3}
  &[ 4, 2,  0, -2 ] = 12 &\qquad
  &[ 4, 2,  0, -4 ] =  3 &\qquad
  &[ 4, 2, -2, -4 ] =  2 \\
  &[ 4, 0, -2, -4 ] =  3 &\qquad
  &[ 2, 0, -2, -4 ] = 12
  \end{alignat*}
  \caption{Quaternary algebra structure on $V(4)$, integral version}
  \label{structureconstants4}
  }
  \end{table}


  \begin{table} \small
  {
  \[
  \begin{array}{rrrrrrrrrrrrrrrrrrrrr}
  12 &
  [6,&\!\!\!\!\!\! 4,&\!\!\!\!\!\! 2,&\!\!\!\!\!\! 0]
  \\
  10 &
  [6,&\!\!\!\!\!\! 4,&\!\!\!\!\!\! 2,&\!\!\!\!\!\! -2]
  \\
  8 &
  [6,&\!\!\!\!\!\! 4,&\!\!\!\!\!\! 2,&\!\!\!\!\!\! -4]&
  [6,&\!\!\!\!\!\! 4,&\!\!\!\!\!\! 0,&\!\!\!\!\!\! -2]
  \\
  6 &
  [6,&\!\!\!\!\!\! 4,&\!\!\!\!\!\! 2,&\!\!\!\!\!\! -6]&
  [6,&\!\!\!\!\!\! 4,&\!\!\!\!\!\! 0,&\!\!\!\!\!\! -4]&
  [6,&\!\!\!\!\!\! 2,&\!\!\!\!\!\! 0,&\!\!\!\!\!\! -2]
  \\
  4 &
  [6,&\!\!\!\!\!\!4,&\!\!\!\!\!\! 0,&\!\!\!\!\!\! -6]&
  [6,&\!\!\!\!\!\! 4,&\!\!\!\!\!\! -2,&\!\!\!\!\!\! -4]&
  [6,&\!\!\!\!\!\! 2,&\!\!\!\!\!\! 0,&\!\!\!\!\!\! -4]&
  [4,&\!\!\!\!\!\! 2,&\!\!\!\!\!\! 0,&\!\!\!\!\!\!-2]
  \\
  2 &
  [6,&\!\!\!\!\!\! 4,&\!\!\!\!\!\! -2,&\!\!\!\!\!\! -6]&
  [6,&\!\!\!\!\!\! 2,&\!\!\!\!\!\! 0,&\!\!\!\!\!\! -6]&
  [6,&\!\!\!\!\!\! 2,&\!\!\!\!\!\! -2,&\!\!\!\!\!\! -4]&
  [4,&\!\!\!\!\!\! 2,&\!\!\!\!\!\! 0,&\!\!\!\!\!\! -4]
  \\
  0 &
  [6,&\!\!\!\!\!\!4,&\!\!\!\!\!\! -4,&\!\!\!\!\!\! -6]&
  [6,&\!\!\!\!\!\! 2,&\!\!\!\!\!\! -2,&\!\!\!\!\!\!-6]&
  [6,&\!\!\!\!\!\! 0,&\!\!\!\!\!\! -2,&\!\!\!\!\!\! -4]&
  [4,&\!\!\!\!\!\! 2,&\!\!\!\!\!\! 0,&\!\!\!\!\!\! -6]&
  [4,&\!\!\!\!\!\! 2,&\!\!\!\!\!\! -2,&\!\!\!\!\!\! -4]
  \\
  -2 &
  [6,&\!\!\!\!\!\! 2,&\!\!\!\!\!\! -4,&\!\!\!\!\!\! -6]&
  [6,&\!\!\!\!\!\! 0,&\!\!\!\!\!\! -2,&\!\!\!\!\!\! -6]&
  [4,&\!\!\!\!\!\! 2,&\!\!\!\!\!\! -2,&\!\!\!\!\!\! -6]&
  [4,&\!\!\!\!\!\! 0,&\!\!\!\!\!\! -2,&\!\!\!\!\!\!-4]
  \\
  -4 &
  [6,&\!\!\!\!\!\! 0,&\!\!\!\!\!\! -4,&\!\!\!\!\!\! -6]&
  [4,&\!\!\!\!\!\! 2,&\!\!\!\!\!\! -4,&\!\!\!\!\!\! -6]&
  [4,&\!\!\!\!\!\! 0,&\!\!\!\!\!\! -2,&\!\!\!\!\!\! -6]&
  [2,&\!\!\!\!\!\! 0,&\!\!\!\!\!\! -2,&\!\!\!\!\!\! -4]
  \\
  -6 &
  [6,&\!\!\!\!\!\! -2,&\!\!\!\!\!\! -4,&\!\!\!\!\!\! -6]&
  [4,&\!\!\!\!\!\! 0,&\!\!\!\!\!\! -4,&\!\!\!\!\!\! -6]&
  [2,&\!\!\!\!\!\! 0,&\!\!\!\!\!\! -2,&\!\!\!\!\!\! -6]
  \\
  -8 &
  [4,&\!\!\!\!\!\! -2,&\!\!\!\!\!\! -4,&\!\!\!\!\!\!-6]&
  [2,&\!\!\!\!\!\! 0,&\!\!\!\!\!\! -4,&\!\!\!\!\!\!-6]
  \\
  -10 &
  [2,&\!\!\!\!\!\! -2,&\!\!\!\!\!\! -4,&\!\!\!\!\!\! -6]
  \\
  -12 &
  [0,&\!\!\!\!\!\! -2,&\!\!\!\!\!\! -4,&\!\!\!\!\!\! -6]
  \\
  \end{array}
  \]
  \caption{Tensor basis of $\Lambda^4 V(6)$}
  \label{tensorbasistable6}
  }
  {
  \begin{alignat*}{2}
  V(12)
  &\left\{
  \begin{array}{rl}
  12\colon &  [1] \\
  10\colon &  [4] \\
   8\colon &  [10, 6] \\
   6\colon &  [20, 20, 4] \\
   4\colon &  [45, 20, 15, 1] \\
   2\colon &  [60, 36, 20, 4] \\
   0\colon &  [50, 64, 10, 10, 6] \\
  -2\colon &  [60, 36, 20, 4] \\
  -4\colon &  [45, 20, 15, 1] \\
  -6\colon &  [20, 20, 4] \\
  -8\colon &  [10, 6] \\
  -10\colon &  [4] \\
  -12\colon &  [1]
  \end{array}
  \right.
  &\qquad
  V(8)
  &\left\{
  \begin{array}{rl}
   8\colon &  [-2, 1] \\
   6\colon &  [-12, -1, 2] \\
   4\colon &  [-21, -2, 4, 1] \\
   2\colon &  [-32, -6, 4, 3] \\
   0\colon &  [-40, -16, 3, 3, 4] \\
  -2\colon &  [-32, -6, 4, 3] \\
  -4\colon &  [-21, -2, 4, 1] \\
  -6\colon &  [-12, -1, 2] \\
  -8\colon &  [-2, 1]
  \end{array}
  \right.
  \\[4pt]
  V(6)
  &\left\{
  \begin{array}{rl}
   6\colon &  [20, -5, 2] \\
   4\colon &  [30, -20, 0, 2] \\
   2\colon &  [0, 30, -20, 5] \\
   0\colon &  [0, 0, -20, 20, 0] \\
  -2\colon &  [0, -30, 20, -5] \\
  -4\colon &  [-30, 20, 0, -2] \\
  -6\colon &  [-20, 5, -2]
  \end{array}
  \right.
  &\qquad
  V(4)
  &\left\{
  \begin{array}{rl}
   4\colon &  [0, 5, -3, 1] \\
   2\colon &  [30, -18, -2, 2] \\
   0\colon &  [75, -12, -3, -3, 3] \\
  -2\colon &  [30, -18, -2, 2] \\
  -4\colon &  [0, 5, -3, 1]
  \end{array}
  \right.
  \\[4pt]
  V(0)
  &\left\{
  \begin{array}{rl}
   0\colon &  [-15, 6, -3, -3, 1]
  \end{array}
  \right.
  \end{alignat*}
  \caption{Weight vector basis of $\Lambda^4 V(6)$}
  \label{weightvectorbasistable6}
  }
  {
  \[
  \begin{array}{llll}
  {[   6,   4,   2,  -6 ]}  = 3 &
  {[   6,   4,   0,  -4 ]}  = -6 &
  {[   6,   2,   0,  -2 ]}  = 15 &
  {[   6,   4,   0,  -6 ]}  = 1
  \\
  {[   6,   4,  -2,  -4 ]}  = -3 &
  {[   6,   2,   0,  -4 ]}  = 0  &
  {[   4,   2,   0,  -2 ]}  = 15 &
  {[   6,   4,  -2,  -6 ]}  = 0
  \\
  {[   6,   2,   0,  -6 ]}  = 1 &
  {[   6,   2,  -2,  -4 ]}  = -3 &
  {[   4,   2,   0,  -4 ]}  = 6 &
  {[   6,   4,  -4,  -6 ]}  = 0
  \\
  {[   6,   2,  -2,  -6 ]}  = 0 &
  {[   6,   0,  -2,  -4 ]}  = -3 &
  {[   4,   2,   0,  -6 ]}  = 3 &
  {[   4,   2,  -2,  -4 ]}  = 0
  \\
  {[   6,   2,  -4,  -6 ]}  = 0 &
  {[   6,   0,  -2,  -6 ]}  = -1 &
  {[   4,   2,  -2,  -6 ]}  = 3 &
  {[   4,   0,  -2,  -4 ]}  = -6
  \\
  {[   6,   0,  -4,  -6 ]}  = -1 &
  {[   4,   2,  -4,  -6 ]}  = 3 &
  {[   4,   0,  -2,  -6 ]}  = 0 &
  {[   2,   0,  -2,  -4 ]}  = -15
  \\
  {[   6,  -2,  -4,  -6 ]}  = -3 &
  {[   4,   0,  -4,  -6 ]}  = 6 &
  {[   2,   0,  -2,  -6 ]}  = -15
  \end{array}
  \]
  \caption{Quaternary algebra structure on $V(6)$, integral version}
  \label{structureconstants6}
  }
  \end{table}


  \begin{table} \tiny
  \[
  \begin{array}{rrrrrrrrrrrrrrrrrrrrrrrrr}
  20
  &\!\!\!\!
  [   8,&\!\!\!\!\!\!\!\!   6,&\!\!\!\!\!\!\!\!   4,&\!\!\!\!\!\!\!\!   2 ]
  \\
  18
  &\!\!\!\!
  [   8,&\!\!\!\!\!\!\!\!   6,&\!\!\!\!\!\!\!\!   4,&\!\!\!\!\!\!\!\!   0 ]
  \\
  16
  &\!\!\!\!
  [   8,&\!\!\!\!\!\!\!\!   6,&\!\!\!\!\!\!\!\!   4,&\!\!\!\!\!\!\!\!  -2 ] &
  [   8,&\!\!\!\!\!\!\!\!   6,&\!\!\!\!\!\!\!\!   2,&\!\!\!\!\!\!\!\!   0 ]
  \\
  14
  &\!\!\!\!
  [   8,&\!\!\!\!\!\!\!\!   6,&\!\!\!\!\!\!\!\!   4,&\!\!\!\!\!\!\!\!  -4 ] &
  [   8,&\!\!\!\!\!\!\!\!   6,&\!\!\!\!\!\!\!\!   2,&\!\!\!\!\!\!\!\!  -2 ] &
  [   8,&\!\!\!\!\!\!\!\!   4,&\!\!\!\!\!\!\!\!   2,&\!\!\!\!\!\!\!\!   0 ]
  \\
  12
  &\!\!\!\!
  [   8,&\!\!\!\!\!\!\!\!   6,&\!\!\!\!\!\!\!\!   4,&\!\!\!\!\!\!\!\!  -6 ] &
  [   8,&\!\!\!\!\!\!\!\!   6,&\!\!\!\!\!\!\!\!   2,&\!\!\!\!\!\!\!\!  -4 ] &
  [   8,&\!\!\!\!\!\!\!\!   6,&\!\!\!\!\!\!\!\!   0,&\!\!\!\!\!\!\!\!  -2 ] &
  [   8,&\!\!\!\!\!\!\!\!   4,&\!\!\!\!\!\!\!\!   2,&\!\!\!\!\!\!\!\!  -2 ] &
  [   6,&\!\!\!\!\!\!\!\!   4,&\!\!\!\!\!\!\!\!   2,&\!\!\!\!\!\!\!\!   0 ]
  \\
  10
  &\!\!\!\!
  [   8,&\!\!\!\!\!\!\!\!   6,&\!\!\!\!\!\!\!\!   4,&\!\!\!\!\!\!\!\!  -8 ] &
  [   8,&\!\!\!\!\!\!\!\!   6,&\!\!\!\!\!\!\!\!   2,&\!\!\!\!\!\!\!\!  -6 ] &
  [   8,&\!\!\!\!\!\!\!\!   6,&\!\!\!\!\!\!\!\!   0,&\!\!\!\!\!\!\!\!  -4 ] &
  [   8,&\!\!\!\!\!\!\!\!   4,&\!\!\!\!\!\!\!\!   2,&\!\!\!\!\!\!\!\!  -4 ] &
  [   8,&\!\!\!\!\!\!\!\!   4,&\!\!\!\!\!\!\!\!   0,&\!\!\!\!\!\!\!\!  -2 ] &
  [   6,&\!\!\!\!\!\!\!\!   4,&\!\!\!\!\!\!\!\!   2,&\!\!\!\!\!\!\!\!  -2 ]
  \\
  8
  &\!\!\!\!
  [   8,&\!\!\!\!\!\!\!\!   6,&\!\!\!\!\!\!\!\!   2,&\!\!\!\!\!\!\!\!  -8 ] &
  [   8,&\!\!\!\!\!\!\!\!   6,&\!\!\!\!\!\!\!\!   0,&\!\!\!\!\!\!\!\!  -6 ] &
  [   8,&\!\!\!\!\!\!\!\!   6,&\!\!\!\!\!\!\!\!  -2,&\!\!\!\!\!\!\!\!  -4 ] &
  [   8,&\!\!\!\!\!\!\!\!   4,&\!\!\!\!\!\!\!\!   2,&\!\!\!\!\!\!\!\!  -6 ] &
  [   8,&\!\!\!\!\!\!\!\!   4,&\!\!\!\!\!\!\!\!   0,&\!\!\!\!\!\!\!\!  -4 ] &
  [   8,&\!\!\!\!\!\!\!\!   2,&\!\!\!\!\!\!\!\!   0,&\!\!\!\!\!\!\!\!  -2 ] \\
  &\!\!\!\!
  [   6,&\!\!\!\!\!\!\!\!   4,&\!\!\!\!\!\!\!\!   2,&\!\!\!\!\!\!\!\!  -4 ] &
  [   6,&\!\!\!\!\!\!\!\!   4,&\!\!\!\!\!\!\!\!   0,&\!\!\!\!\!\!\!\!  -2 ]
  \\
  6
  &\!\!\!\!
  [   8,&\!\!\!\!\!\!\!\!   6,&\!\!\!\!\!\!\!\!   0,&\!\!\!\!\!\!\!\!  -8 ] &
  [   8,&\!\!\!\!\!\!\!\!   6,&\!\!\!\!\!\!\!\!  -2,&\!\!\!\!\!\!\!\!  -6 ] &
  [   8,&\!\!\!\!\!\!\!\!   4,&\!\!\!\!\!\!\!\!   2,&\!\!\!\!\!\!\!\!  -8 ] &
  [   8,&\!\!\!\!\!\!\!\!   4,&\!\!\!\!\!\!\!\!   0,&\!\!\!\!\!\!\!\!  -6 ] &
  [   8,&\!\!\!\!\!\!\!\!   4,&\!\!\!\!\!\!\!\!  -2,&\!\!\!\!\!\!\!\!  -4 ] &
  [   8,&\!\!\!\!\!\!\!\!   2,&\!\!\!\!\!\!\!\!   0,&\!\!\!\!\!\!\!\!  -4 ] \\
  &\!\!\!\!
  [   6,&\!\!\!\!\!\!\!\!   4,&\!\!\!\!\!\!\!\!   2,&\!\!\!\!\!\!\!\!  -6 ] &
  [   6,&\!\!\!\!\!\!\!\!   4,&\!\!\!\!\!\!\!\!   0,&\!\!\!\!\!\!\!\!  -4 ] &
  [   6,&\!\!\!\!\!\!\!\!   2,&\!\!\!\!\!\!\!\!   0,&\!\!\!\!\!\!\!\!  -2 ]
  \\
  4
  &\!\!\!\!
  [   8,&\!\!\!\!\!\!\!\!   6,&\!\!\!\!\!\!\!\!  -2,&\!\!\!\!\!\!\!\!  -8 ] &
  [   8,&\!\!\!\!\!\!\!\!   6,&\!\!\!\!\!\!\!\!  -4,&\!\!\!\!\!\!\!\!  -6 ] &
  [   8,&\!\!\!\!\!\!\!\!   4,&\!\!\!\!\!\!\!\!   0,&\!\!\!\!\!\!\!\!  -8 ] &
  [   8,&\!\!\!\!\!\!\!\!   4,&\!\!\!\!\!\!\!\!  -2,&\!\!\!\!\!\!\!\!  -6 ] &
  [   8,&\!\!\!\!\!\!\!\!   2,&\!\!\!\!\!\!\!\!   0,&\!\!\!\!\!\!\!\!  -6 ] &
  [   8,&\!\!\!\!\!\!\!\!   2,&\!\!\!\!\!\!\!\!  -2,&\!\!\!\!\!\!\!\!  -4 ] \\
  &\!\!\!\!
  [   6,&\!\!\!\!\!\!\!\!   4,&\!\!\!\!\!\!\!\!   2,&\!\!\!\!\!\!\!\!  -8 ] &
  [   6,&\!\!\!\!\!\!\!\!   4,&\!\!\!\!\!\!\!\!   0,&\!\!\!\!\!\!\!\!  -6 ] &
  [   6,&\!\!\!\!\!\!\!\!   4,&\!\!\!\!\!\!\!\!  -2,&\!\!\!\!\!\!\!\!  -4 ] &
  [   6,&\!\!\!\!\!\!\!\!   2,&\!\!\!\!\!\!\!\!   0,&\!\!\!\!\!\!\!\!  -4 ] &
  [   4,&\!\!\!\!\!\!\!\!   2,&\!\!\!\!\!\!\!\!   0,&\!\!\!\!\!\!\!\!  -2 ]
  \\
  2
  &\!\!\!\!
  [   8,&\!\!\!\!\!\!\!\!   6,&\!\!\!\!\!\!\!\!  -4,&\!\!\!\!\!\!\!\!  -8 ] &
  [   8,&\!\!\!\!\!\!\!\!   4,&\!\!\!\!\!\!\!\!  -2,&\!\!\!\!\!\!\!\!  -8 ] &
  [   8,&\!\!\!\!\!\!\!\!   4,&\!\!\!\!\!\!\!\!  -4,&\!\!\!\!\!\!\!\!  -6 ] &
  [   8,&\!\!\!\!\!\!\!\!   2,&\!\!\!\!\!\!\!\!   0,&\!\!\!\!\!\!\!\!  -8 ] &
  [   8,&\!\!\!\!\!\!\!\!   2,&\!\!\!\!\!\!\!\!  -2,&\!\!\!\!\!\!\!\!  -6 ] &
  [   8,&\!\!\!\!\!\!\!\!   0,&\!\!\!\!\!\!\!\!  -2,&\!\!\!\!\!\!\!\!  -4 ] \\
  &\!\!\!\!
  [   6,&\!\!\!\!\!\!\!\!   4,&\!\!\!\!\!\!\!\!   0,&\!\!\!\!\!\!\!\!  -8 ] &
  [   6,&\!\!\!\!\!\!\!\!   4,&\!\!\!\!\!\!\!\!  -2,&\!\!\!\!\!\!\!\!  -6 ] &
  [   6,&\!\!\!\!\!\!\!\!   2,&\!\!\!\!\!\!\!\!   0,&\!\!\!\!\!\!\!\!  -6 ] &
  [   6,&\!\!\!\!\!\!\!\!   2,&\!\!\!\!\!\!\!\!  -2,&\!\!\!\!\!\!\!\!  -4 ] &
  [   4,&\!\!\!\!\!\!\!\!   2,&\!\!\!\!\!\!\!\!   0,&\!\!\!\!\!\!\!\!  -4 ]
  \\
  0
  &\!\!\!\!
  [   8,&\!\!\!\!\!\!\!\!   6,&\!\!\!\!\!\!\!\!  -6,&\!\!\!\!\!\!\!\!  -8 ] &
  [   8,&\!\!\!\!\!\!\!\!   4,&\!\!\!\!\!\!\!\!  -4,&\!\!\!\!\!\!\!\!  -8 ] &
  [   8,&\!\!\!\!\!\!\!\!   2,&\!\!\!\!\!\!\!\!  -2,&\!\!\!\!\!\!\!\!  -8 ] &
  [   8,&\!\!\!\!\!\!\!\!   2,&\!\!\!\!\!\!\!\!  -4,&\!\!\!\!\!\!\!\!  -6 ] &
  [   8,&\!\!\!\!\!\!\!\!   0,&\!\!\!\!\!\!\!\!  -2,&\!\!\!\!\!\!\!\!  -6 ] &
  [   6,&\!\!\!\!\!\!\!\!   4,&\!\!\!\!\!\!\!\!  -2,&\!\!\!\!\!\!\!\!  -8 ] \\
  &\!\!\!\!
  [   6,&\!\!\!\!\!\!\!\!   4,&\!\!\!\!\!\!\!\!  -4,&\!\!\!\!\!\!\!\!  -6 ] &
  [   6,&\!\!\!\!\!\!\!\!   2,&\!\!\!\!\!\!\!\!   0,&\!\!\!\!\!\!\!\!  -8 ] &
  [   6,&\!\!\!\!\!\!\!\!   2,&\!\!\!\!\!\!\!\!  -2,&\!\!\!\!\!\!\!\!  -6 ] &
  [   6,&\!\!\!\!\!\!\!\!   0,&\!\!\!\!\!\!\!\!  -2,&\!\!\!\!\!\!\!\!  -4 ] &
  [   4,&\!\!\!\!\!\!\!\!   2,&\!\!\!\!\!\!\!\!   0,&\!\!\!\!\!\!\!\!  -6 ] &
  [   4,&\!\!\!\!\!\!\!\!   2,&\!\!\!\!\!\!\!\!  -2,&\!\!\!\!\!\!\!\!  -4 ]
  \\
  -2
  &\!\!\!\!
  [   8,&\!\!\!\!\!\!\!\!   4,&\!\!\!\!\!\!\!\!  -6,&\!\!\!\!\!\!\!\!  -8 ] &
  [   8,&\!\!\!\!\!\!\!\!   2,&\!\!\!\!\!\!\!\!  -4,&\!\!\!\!\!\!\!\!  -8 ] &
  [   8,&\!\!\!\!\!\!\!\!   0,&\!\!\!\!\!\!\!\!  -2,&\!\!\!\!\!\!\!\!  -8 ] &
  [   8,&\!\!\!\!\!\!\!\!   0,&\!\!\!\!\!\!\!\!  -4,&\!\!\!\!\!\!\!\!  -6 ] &
  [   6,&\!\!\!\!\!\!\!\!   4,&\!\!\!\!\!\!\!\!  -4,&\!\!\!\!\!\!\!\!  -8 ] &
  [   6,&\!\!\!\!\!\!\!\!   2,&\!\!\!\!\!\!\!\!  -2,&\!\!\!\!\!\!\!\!  -8 ] \\
  &\!\!\!\!
  [   6,&\!\!\!\!\!\!\!\!   2,&\!\!\!\!\!\!\!\!  -4,&\!\!\!\!\!\!\!\!  -6 ] &
  [   6,&\!\!\!\!\!\!\!\!   0,&\!\!\!\!\!\!\!\!  -2,&\!\!\!\!\!\!\!\!  -6 ] &
  [   4,&\!\!\!\!\!\!\!\!   2,&\!\!\!\!\!\!\!\!   0,&\!\!\!\!\!\!\!\!  -8 ] &
  [   4,&\!\!\!\!\!\!\!\!   2,&\!\!\!\!\!\!\!\!  -2,&\!\!\!\!\!\!\!\!  -6 ] &
  [   4,&\!\!\!\!\!\!\!\!   0,&\!\!\!\!\!\!\!\!  -2,&\!\!\!\!\!\!\!\!  -4 ]
  \\
  -4
  &\!\!\!\!
  [   8,&\!\!\!\!\!\!\!\!   2,&\!\!\!\!\!\!\!\!  -6,&\!\!\!\!\!\!\!\!  -8 ] &
  [   8,&\!\!\!\!\!\!\!\!   0,&\!\!\!\!\!\!\!\!  -4,&\!\!\!\!\!\!\!\!  -8 ] &
  [   8,&\!\!\!\!\!\!\!\!  -2,&\!\!\!\!\!\!\!\!  -4,&\!\!\!\!\!\!\!\!  -6 ] &
  [   6,&\!\!\!\!\!\!\!\!   4,&\!\!\!\!\!\!\!\!  -6,&\!\!\!\!\!\!\!\!  -8 ] &
  [   6,&\!\!\!\!\!\!\!\!   2,&\!\!\!\!\!\!\!\!  -4,&\!\!\!\!\!\!\!\!  -8 ] &
  [   6,&\!\!\!\!\!\!\!\!   0,&\!\!\!\!\!\!\!\!  -2,&\!\!\!\!\!\!\!\!  -8 ] \\
  &\!\!\!\!
  [   6,&\!\!\!\!\!\!\!\!   0,&\!\!\!\!\!\!\!\!  -4,&\!\!\!\!\!\!\!\!  -6 ] &
  [   4,&\!\!\!\!\!\!\!\!   2,&\!\!\!\!\!\!\!\!  -2,&\!\!\!\!\!\!\!\!  -8 ] &
  [   4,&\!\!\!\!\!\!\!\!   2,&\!\!\!\!\!\!\!\!  -4,&\!\!\!\!\!\!\!\!  -6 ] &
  [   4,&\!\!\!\!\!\!\!\!   0,&\!\!\!\!\!\!\!\!  -2,&\!\!\!\!\!\!\!\!  -6 ] &
  [   2,&\!\!\!\!\!\!\!\!   0,&\!\!\!\!\!\!\!\!  -2,&\!\!\!\!\!\!\!\!  -4 ]
  \\
  -6
  &\!\!\!\!
  [   8,&\!\!\!\!\!\!\!\!   0,&\!\!\!\!\!\!\!\!  -6,&\!\!\!\!\!\!\!\!  -8 ] &
  [   8,&\!\!\!\!\!\!\!\!  -2,&\!\!\!\!\!\!\!\!  -4,&\!\!\!\!\!\!\!\!  -8 ] &
  [   6,&\!\!\!\!\!\!\!\!   2,&\!\!\!\!\!\!\!\!  -6,&\!\!\!\!\!\!\!\!  -8 ] &
  [   6,&\!\!\!\!\!\!\!\!   0,&\!\!\!\!\!\!\!\!  -4,&\!\!\!\!\!\!\!\!  -8 ] &
  [   6,&\!\!\!\!\!\!\!\!  -2,&\!\!\!\!\!\!\!\!  -4,&\!\!\!\!\!\!\!\!  -6 ] &
  [   4,&\!\!\!\!\!\!\!\!   2,&\!\!\!\!\!\!\!\!  -4,&\!\!\!\!\!\!\!\!  -8 ] \\
  &\!\!\!\!
  [   4,&\!\!\!\!\!\!\!\!   0,&\!\!\!\!\!\!\!\!  -2,&\!\!\!\!\!\!\!\!  -8 ] &
  [   4,&\!\!\!\!\!\!\!\!   0,&\!\!\!\!\!\!\!\!  -4,&\!\!\!\!\!\!\!\!  -6 ] &
  [   2,&\!\!\!\!\!\!\!\!   0,&\!\!\!\!\!\!\!\!  -2,&\!\!\!\!\!\!\!\!  -6 ]
  \\
  -8
  &\!\!\!\!
  [   8,&\!\!\!\!\!\!\!\!  -2,&\!\!\!\!\!\!\!\!  -6,&\!\!\!\!\!\!\!\!  -8 ] &
  [   6,&\!\!\!\!\!\!\!\!   0,&\!\!\!\!\!\!\!\!  -6,&\!\!\!\!\!\!\!\!  -8 ] &
  [   6,&\!\!\!\!\!\!\!\!  -2,&\!\!\!\!\!\!\!\!  -4,&\!\!\!\!\!\!\!\!  -8 ] &
  [   4,&\!\!\!\!\!\!\!\!   2,&\!\!\!\!\!\!\!\!  -6,&\!\!\!\!\!\!\!\!  -8 ] &
  [   4,&\!\!\!\!\!\!\!\!   0,&\!\!\!\!\!\!\!\!  -4,&\!\!\!\!\!\!\!\!  -8 ] &
  [   4,&\!\!\!\!\!\!\!\!  -2,&\!\!\!\!\!\!\!\!  -4,&\!\!\!\!\!\!\!\!  -6 ] \\
  &\!\!\!\!
  [   2,&\!\!\!\!\!\!\!\!   0,&\!\!\!\!\!\!\!\!  -2,&\!\!\!\!\!\!\!\!  -8 ] &
  [   2,&\!\!\!\!\!\!\!\!   0,&\!\!\!\!\!\!\!\!  -4,&\!\!\!\!\!\!\!\!  -6 ]
  \\
  -10
  &\!\!\!\!
  [   8,&\!\!\!\!\!\!\!\!  -4,&\!\!\!\!\!\!\!\!  -6,&\!\!\!\!\!\!\!\!  -8 ] &
  [   6,&\!\!\!\!\!\!\!\!  -2,&\!\!\!\!\!\!\!\!  -6,&\!\!\!\!\!\!\!\!  -8 ] &
  [   4,&\!\!\!\!\!\!\!\!   0,&\!\!\!\!\!\!\!\!  -6,&\!\!\!\!\!\!\!\!  -8 ] &
  [   4,&\!\!\!\!\!\!\!\!  -2,&\!\!\!\!\!\!\!\!  -4,&\!\!\!\!\!\!\!\!  -8 ] &
  [   2,&\!\!\!\!\!\!\!\!   0,&\!\!\!\!\!\!\!\!  -4,&\!\!\!\!\!\!\!\!  -8 ] &
  [   2,&\!\!\!\!\!\!\!\!  -2,&\!\!\!\!\!\!\!\!  -4,&\!\!\!\!\!\!\!\!  -6 ]
  \\
  -12
  &\!\!\!\!
  [   6,&\!\!\!\!\!\!\!\!  -4,&\!\!\!\!\!\!\!\!  -6,&\!\!\!\!\!\!\!\!  -8 ] &
  [   4,&\!\!\!\!\!\!\!\!  -2,&\!\!\!\!\!\!\!\!  -6,&\!\!\!\!\!\!\!\!  -8 ] &
  [   2,&\!\!\!\!\!\!\!\!   0,&\!\!\!\!\!\!\!\!  -6,&\!\!\!\!\!\!\!\!  -8 ] &
  [   2,&\!\!\!\!\!\!\!\!  -2,&\!\!\!\!\!\!\!\!  -4,&\!\!\!\!\!\!\!\!  -8 ] &
  [   0,&\!\!\!\!\!\!\!\!  -2,&\!\!\!\!\!\!\!\!  -4,&\!\!\!\!\!\!\!\!  -6 ]
  \\
  -14
  &\!\!\!\!
  [   4,&\!\!\!\!\!\!\!\!  -4,&\!\!\!\!\!\!\!\!  -6,&\!\!\!\!\!\!\!\!  -8 ] &
  [   2,&\!\!\!\!\!\!\!\!  -2,&\!\!\!\!\!\!\!\!  -6,&\!\!\!\!\!\!\!\!  -8 ] &
  [   0,&\!\!\!\!\!\!\!\!  -2,&\!\!\!\!\!\!\!\!  -4,&\!\!\!\!\!\!\!\!  -8 ]
  \\
  -16
  &\!\!\!\!
  [   2,&\!\!\!\!\!\!\!\!  -4,&\!\!\!\!\!\!\!\!  -6,&\!\!\!\!\!\!\!\!  -8 ] &
  [   0,&\!\!\!\!\!\!\!\!  -2,&\!\!\!\!\!\!\!\!  -6,&\!\!\!\!\!\!\!\!  -8 ]
  \\
  -18
  &\!\!\!\!
  [   0,&\!\!\!\!\!\!\!\!  -4,&\!\!\!\!\!\!\!\!  -6,&\!\!\!\!\!\!\!\!  -8 ]
  \\
  -20
  &\!\!\!\!
  [  -2,&\!\!\!\!\!\!\!\!  -4,&\!\!\!\!\!\!\!\!  -6,&\!\!\!\!\!\!\!\!  -8 ]
  \end{array}
  \]
  \caption{Tensor basis of $\Lambda^4 V(8)$}
  \label{tensorbasistable8}
  \end{table}

  \begin{table} \tiny
  \begin{align*}
  V(20)
  &\left\{
  \begin{array}{rl}
  20\colon & [1] \\
  18\colon & [4] \\
  16\colon & [10,6] \\
  14\colon & [20,20,4] \\
  12\colon & [35,45,20,15,1] \\
  10\colon & [56,84,60,36,20,4] \\
  8\colon & [140,126,50,70,64,10,10,6] \\
  6\colon & [224,140,120,140,60,36,20,20,4] \\
  4\colon & [280,105,256,175,84,45,35,45,20,15,1] \\
  2\colon & [280,360,140,160,140,20,84,60,36,20,4] \\
  0\colon & [196,384,300,126,70,126,50,70,64,10,10,6] \\
  -2\colon & [280,360,160,84,140,140,60,36,20,20,4] \\
  -4\colon & [280,256,35,105,175,84,45,45,20,15,1] \\
  -6\colon & [224,120,140,140,20,60,36,20,4] \\
  -8\colon & [140,126,70,50,64,10,10,6] \\
  -10\colon & [56,84,60,36,20,4] \\
  -12\colon & [35,45,20,15,1] \\
  -14\colon & [20,20,4] \\
  -16\colon & [10,6] \\
  -18\colon & [4] \\
  -20\colon & [1]
  \end{array}
  \right.
  \\[2pt]
  V(16)
  &\left\{
  \begin{array}{rl}
  16\colon & [-3,2] \\
  14\colon & [-18,1,4] \\
  12\colon & [-63,-24,2,11,2] \\
  10\colon & [-168,-119,-28,6,16,7] \\
  8\colon & [-364,-168,-35,-49,16,12,12,11] \\
  6\colon & [-560,-217,-224,-84,2,24,7,26,9] \\
  4\colon & [-756,-217,-448,-140,-14,21,-28,21,22,26,3] \\
  2\colon & [-896,-696,-182,-208,-49,12,-56,17,33,31,10] \\
  0\colon & [-784,-928,-440,-105,-14,-105,-10,-14,48,17,17,14] \\
  -2\colon & [-896,-696,-208,-56,-182,-49,17,33,12,31,10] \\
  -4\colon & [-756,-448,-28,-217,-140,-14,21,21,22,26,3] \\
  -6\colon & [-560,-224,-217,-84,7,2,24,26,9] \\
  -8\colon & [-364,-168,-49,-35,16,12,12,11] \\
  -10\colon & [-168,-119,-28,6,16,7] \\
  -12\colon & [-63,-24,2,11,2] \\
  -14\colon & [-18,1,4] \\
  -16\colon & [-3,2]
  \end{array}
  \right.
  \end{align*}
  \caption{Weight vector basis of $\Lambda^4 V(8)$: part one}
  \label{weightvectorbasistable8part1}
  \end{table}

  \begin{table} \tiny
  \begin{align*}
  V(14)
  &\left\{
  \begin{array}{rl}
  14\colon & [14,-7,4] \\
  12\colon & [98,0,-28,6,4] \\
  10\colon & [392,147,-84,18,-16,13] \\
  8\colon & [784,0,-140,140,-64,-16,32,12] \\
  6\colon & [784,-245,672,28,-150,-72,91,34,5] \\
  4\colon & [392,-294,896,-280,-84,-162,280,126,4,12,2] \\
  2\colon & [0,504,-378,336,-399,-108,504,63,63,-15,6] \\
  0\colon & [0,0,0,-504,-336,504,0,336,0,-24,24,0] \\
  -2\colon & [0,-504,-336,-504,378,399,-63,-63,108,15,-6] \\
  -4\colon & [-392,-896,-280,294,280,84,-126,162,-4,-12,-2] \\
  -6\colon & [-784,-672,245,-28,-91,150,72,-34,-5] \\
  -8\colon & [-784,0,-140,140,64,-32,16,-12] \\
  -10\colon & [-392,-147,84,-18,16,-13] \\
  -12\colon & [-98,0,28,-6,-4] \\
  -14\colon & [-14,7,-4]
  \end{array}
  \right.
  \\[2pt]
  V(12)'
  &\left\{
  \begin{array}{rl}
  12\colon & [15,-5,3,0,0] \\
  10\colon & [120,10,-2,-10,6,0] \\
  8\colon & [220,13,-5,-25,-4,9,-5,3] \\
  6\colon & [328,10,80,-34,-10,14,-20,-2,6] \\
  4\colon & [430,15,176,-55,-1,10,-20,-32,-5,11,3] \\
  2\colon & [540,260,-60,104,-20,8,-32,-50,-4,10,8] \\
  0\colon & [630,360,190,-50,-4,-50,-60,-4,-20,8,8,10] \\
  -2\colon & [540,260,104,-32,-60,-20,-50,-4,8,10,8] \\
  -4\colon & [430,176,-20,15,-55,-1,-32,10,-5,11,3] \\
  -6\colon & [328,80,10,-34,-20,-10,14,-2,6] \\
  -8\colon & [220,13,-25,-5,-4,-5,9,3] \\
  -10\colon & [120,10,-2,-10,6,0] \\
  -12\colon & [15,-5,3,0,0]
  \end{array}
  \right.
  \\[2pt]
  V(12)''
  &\left\{
  \begin{array}{rl}
  12\colon & [-42,14,0,-6,3] \\
  10\colon & [-336,-28,56,-8,-24,9] \\
  8\colon & [-616,140,140,-56,-32,-36,23,6] \\
  6\colon & [-448,560,-560,-56,40,-104,35,32,-6] \\
  4\colon & [560,840,-896,280,-224,-100,-70,77,50,-11,-3] \\
  2\colon & [2016,-224,672,-896,-196,-80,-112,203,-14,-1,-8] \\
  0\colon & [2352,1344,-1120,140,-224,140,315,-224,56,-14,-14,-7] \\
  -2\colon & [2016,-224,-896,-112,672,-196,203,-14,-80,-1,-8] \\
  -4\colon & [560,-896,-70,840,280,-224,77,-100,50,-11,-3] \\
  -6\colon & [-448,-560,560,-56,35,40,-104,32,-6] \\
  -8\colon & [-616,140,-56,140,-32,23,-36,6] \\
  -10\colon & [-336,-28,56,-8,-24,9] \\
  -12\colon & [-42,14,0,-6,3]
  \end{array}
  \right.
  \\[2pt]
  V(10)
  &\left\{
  \begin{array}{rl}
  10\colon & [-210,35,-7,-5,3,0] \\
  8\colon & [-350,91,-35,35,-16,9,-5,3] \\
  6\colon & [-336,105,-210,105,-75,3,0,-9,9] \\
  4\colon & [-280,210,-224,70,112,-70,-70,14,-40,10,6] \\
  2\colon & [0,-280,210,56,70,-70,-98,-35,56,-35,14] \\
  0\colon & [0,0,0,210,-42,-210,0,42,0,-42,42,0] \\
  -2\colon & [0,280,-56,98,-210,-70,35,-56,70,35,-14] \\
  -4\colon & [280,224,70,-210,-70,-112,-14,70,40,-10,-6] \\
  -6\colon & [336,210,-105,-105,0,75,-3,9,-9] \\
  -8\colon & [350,-91,-35,35,16,5,-9,-3] \\
  -10\colon & [210,-35,7,5,-3,0]
  \end{array}
  \right.
  \\[2pt]
  V(8)'
  &\left\{
  \begin{array}{rl}
  8\colon & [0,-7,7,5,-2,1,0,0] \\
  6\colon & [-56,14,40,-8,4,0,5,-2,1] \\
  4\colon & [-84,42,-8,8,-12,6,40,-1,-3,0,1] \\
  2\colon & [-56,-48,44,-40,2,8,48,-6,-5,-1,2] \\
  0\colon & [-98,-12,-82,36,16,36,2,16,-12,1,1,2] \\
  -2\colon & [-56,-48,-40,48,44,2,-6,-5,8,-1,2] \\
  -4\colon & [-84,-8,40,42,8,-12,-1,6,-3,0,1] \\
  -6\colon & [-56,40,14,-8,5,4,0,-2,1] \\
  -8\colon & [0,-7,5,7,-2,0,1,0]
  \end{array}
  \right.
  \\[2pt]
  V(8)''
  &\left\{
  \begin{array}{rl}
  8\colon & [0,-56,42,40,-8,0,-5,3] \\
  6\colon & [-448,14,320,-8,44,-24,5,-10,9] \\
  4\colon & [-1064,42,160,148,-96,6,180,-29,-3,0,9] \\
  2\colon & [-2016,-48,324,-96,2,8,216,-6,-61,-1,18] \\
  0\colon & [-3528,-432,-152,246,16,246,72,16,-82,1,1,22] \\
  -2\colon & [-2016,-48,-96,216,324,2,-6,-61,8,-1,18] \\
  -4\colon & [-1064,160,180,42,148,-96,-29,6,-3,0,9] \\
  -6\colon & [-448,320,14,-8,5,44,-24,-10,9] \\
  -8\colon & [0,-56,40,42,-8,-5,0,3]
  \end{array}
  \right.
  \end{align*}
  \caption{Weight vector basis of $\Lambda^4 V(8)$: part two}
  \label{weightvectorbasistable8part2}
  \end{table}

  \begin{table} \tiny
  \begin{align*}
  V(6)
  &\left\{
  \begin{array}{rl}
  6\colon & [224,-70,-160,8,20,-8,15,-6,3] \\
  4\colon & [560,-420,-128,40,-32,20,-40,26,-10,-8,6] \\
  2\colon & [0,400,-300,-320,50,40,-40,50,-5,-25,10] \\
  0\colon & [0,0,0,-200,160,200,0,-160,0,-20,20,0] \\
  -2\colon & [0,-400,320,40,300,-50,-50,5,-40,25,-10] \\
  -4\colon & [-560,128,40,420,-40,32,-26,-20,10,8,-6] \\
  -6\colon & [-224,160,70,-8,-15,-20,8,6,-3]
  \end{array}
  \right.
  \\[2pt]
  V(4)'
  &\left\{
  \begin{array}{rl}
  4\colon & [28,-14,-16,2,4,-2,10,-2,1,0,0] \\
  2\colon & [56,-8,-16,-16,12,-8,8,-1,-2,1,0] \\
  0\colon & [196,-32,-4,12,-4,12,-11,-4,3,-2,-2,1] \\
  -2\colon & [56,-8,-16,8,-16,12,-1,-2,-8,1,0] \\
  -4\colon & [28,-16,10,-14,2,4,-2,-2,1,0,0]
  \end{array}
  \right.
  \\[2pt]
  V(4)''
  &\left\{
  \begin{array}{rl}
  4\colon & [-294,245,168,-63,-14,21,-105,21,0,-7,3] \\
  2\colon & [196,-252,112,392,-112,84,-84,42,0,-14,4] \\
  0\colon & [686,-112,154,-168,70,-168,182,70,-42,14,14,-4] \\
  -2\colon & [196,-252,392,-84,112,-112,42,0,84,-14,4] \\
  -4\colon & [-294,168,-105,245,-63,-14,21,21,0,-7,3]
  \end{array}
  \right.
  \\[2pt]
  V(0)
  &\left\{
  \begin{array}{rl}
  0\colon & [-448,128,-32,-24,16,-24,2,16,3,-4,-4,2]
  \end{array}
  \right.
  \end{align*}
  \caption{Weight vector basis of $\Lambda^4 V(8)$: part three}
  \label{weightvectorbasistable8part3}
  \end{table}


  \begin{table} \tiny
  \begin{center}
  \[
  \begin{array}{llll}
  {[  8,  6,  2, -8 ]}  =       0    &\quad
  {[  8,  6,  0, -6 ]}  =    -168    &\quad
  {[  8,  6, -2, -4 ]}  =    1904    &\quad
  {[  8,  4,  2, -6 ]}  =     336    \\
  {[  8,  4,  0, -4 ]}  =   -4172    &\quad
  {[  8,  2,  0, -2 ]}  =   15512    &\quad
  {[  6,  4,  2, -4 ]}  =   14336    &\quad
  {[  6,  4,  0, -2 ]}  =  -21504    \\
  {[  8,  6,  0, -8 ]}  =     -21    &\quad
  {[  8,  6, -2, -6 ]}  =     392    &\quad
  {[  8,  4,  2, -8 ]}  =      42    &\quad
  {[  8,  4,  0, -6 ]}  =    -980    \\
  {[  8,  4, -2, -4 ]}  =    -420    &\quad
  {[  8,  2,  0, -4 ]}  =    2688    &\quad
  {[  6,  4,  2, -6 ]}  =    3920    &\quad
  {[  6,  4,  0, -4 ]}  =   -3276    \\
  {[  6,  2,  0, -2 ]}  =    -616    &\quad
  {[  8,  6, -2, -8 ]}  =      44    &\quad
  {[  8,  6, -4, -6 ]}  =     168    &\quad
  {[  8,  4,  0, -8 ]}  =    -131    \\
  {[  8,  4, -2, -6 ]}  =    -288    &\quad
  {[  8,  2,  0, -6 ]}  =     -72    &\quad
  {[  8,  2, -2, -4 ]}  =    1176    &\quad
  {[  6,  4,  2, -8 ]}  =     608    \\
  {[  6,  4,  0, -6 ]}  =     744    &\quad
  {[  6,  4, -2, -4 ]}  =   -2352    &\quad
  {[  6,  2,  0, -4 ]}  =       0    &\quad
  {[  4,  2,  0, -2 ]}  =    -616    \\
  {[  8,  6, -4, -8 ]}  =      50    &\quad
  {[  8,  4, -2, -8 ]}  =     -84    &\quad
  {[  8,  4, -4, -6 ]}  =      52    &\quad
  {[  8,  2,  0, -8 ]}  =    -143    \\
  {[  8,  2, -2, -6 ]}  =      56    &\quad
  {[  8,  0, -2, -4 ]}  =     980    &\quad
  {[  6,  4,  0, -8 ]}  =     456    &\quad
  {[  6,  4, -2, -6 ]}  =    -672    \\
  {[  6,  2,  0, -6 ]}  =     648    &\quad
  {[  6,  2, -2, -4 ]}  =    -784    &\quad
  {[  4,  2,  0, -4 ]}  =    -308    &\quad
  {[  8,  6, -6, -8 ]}  =      20    \\
  {[  8,  4, -4, -8 ]}  =      30    &\quad
  {[  8,  2, -2, -8 ]}  =    -204    &\quad
  {[  8,  2, -4, -6 ]}  =      96    &\quad
  {[  8,  0, -2, -6 ]}  =     448    \\
  {[  6,  4, -2, -8 ]}  =      96    &\quad
  {[  6,  4, -4, -6 ]}  =    -320    &\quad
  {[  6,  2,  0, -8 ]}  =     448    &\quad
  {[  6,  2, -2, -6 ]}  =    -512    \\
  {[  6,  0, -2, -4 ]}  =     784    &\quad
  {[  4,  2,  0, -6 ]}  =     784    &\quad
  {[  4,  2, -2, -4 ]}  =   -1344    &\quad
  {[  8,  4, -6, -8 ]}  =      50    \\
  {[  8,  2, -4, -8 ]}  =     -84    &\quad
  {[  8,  0, -2, -8 ]}  =    -143    &\quad
  {[  8,  0, -4, -6 ]}  =     456    &\quad
  {[  6,  4, -4, -8 ]}  =      52    \\
  {[  6,  2, -2, -8 ]}  =      56    &\quad
  {[  6,  2, -4, -6 ]}  =    -672    &\quad
  {[  6,  0, -2, -6 ]}  =     648    &\quad
  {[  4,  2,  0, -8 ]}  =     980    \\
  {[  4,  2, -2, -6 ]}  =    -784    &\quad
  {[  4,  0, -2, -4 ]}  =    -308    &\quad
  {[  8,  2, -6, -8 ]}  =      44    &\quad
  {[  8,  0, -4, -8 ]}  =    -131    \\
  {[  8, -2, -4, -6 ]}  =     608    &\quad
  {[  6,  4, -6, -8 ]}  =     168    &\quad
  {[  6,  2, -4, -8 ]}  =    -288    &\quad
  {[  6,  0, -2, -8 ]}  =     -72    \\
  {[  6,  0, -4, -6 ]}  =     744    &\quad
  {[  4,  2, -2, -8 ]}  =    1176    &\quad
  {[  4,  2, -4, -6 ]}  =   -2352    &\quad
  {[  4,  0, -2, -6 ]}  =       0    \\
  {[  2,  0, -2, -4 ]}  =    -616    &\quad
  {[  8,  0, -6, -8 ]}  =     -21    &\quad
  {[  8, -2, -4, -8 ]}  =      42    &\quad
  {[  6,  2, -6, -8 ]}  =     392    \\
  {[  6,  0, -4, -8 ]}  =    -980    &\quad
  {[  6, -2, -4, -6 ]}  =    3920    &\quad
  {[  4,  2, -4, -8 ]}  =    -420    &\quad
  {[  4,  0, -2, -8 ]}  =    2688    \\
  {[  4,  0, -4, -6 ]}  =   -3276    &\quad
  {[  2,  0, -2, -6 ]}  =    -616    &\quad
  {[  8, -2, -6, -8 ]}  =       0    &\quad
  {[  6,  0, -6, -8 ]}  =    -168    \\
  {[  6, -2, -4, -8 ]}  =     336    &\quad
  {[  4,  2, -6, -8 ]}  =    1904    &\quad
  {[  4,  0, -4, -8 ]}  =   -4172    &\quad
  {[  4, -2, -4, -6 ]}  =   14336    \\
  {[  2,  0, -2, -8 ]}  =   15512    &\quad
  {[  2,  0, -4, -6 ]}  =  -21504    &\quad
  \end{array}
  \]
  \end{center}
  \caption{First quaternary algebra structure $f$ on $V(8)$}
  \label{structureconstants8part1}
  \begin{center}
  \[
  \begin{array}{llll}
  {[  8,  6,  2, -8 ]}  =       0    &\quad
  {[  8,  6,  0, -6 ]}  =    -112    &\quad
  {[  8,  6, -2, -4 ]}  =      56    &\quad
  {[  8,  4,  2, -6 ]}  =     224    \\
  {[  8,  4,  0, -4 ]}  =     252    &\quad
  {[  8,  2,  0, -2 ]}  =   -1792    &\quad
  {[  6,  4,  2, -4 ]}  =   -2576    &\quad
  {[  6,  4,  0, -2 ]}  =    3864    \\
  {[  8,  6,  0, -8 ]}  =     -14    &\quad
  {[  8,  6, -2, -6 ]}  =     -42    &\quad
  {[  8,  4,  2, -8 ]}  =      28    &\quad
  {[  8,  4,  0, -6 ]}  =     105    \\
  {[  8,  4, -2, -4 ]}  =     175    &\quad
  {[  8,  2,  0, -4 ]}  =    -483    &\quad
  {[  6,  4,  2, -6 ]}  =    -420    &\quad
  {[  6,  4,  0, -4 ]}  =      91    \\
  {[  6,  2,  0, -2 ]}  =    1106    &\quad
  {[  8,  6, -2, -8 ]}  =     -14    &\quad
  {[  8,  6, -4, -6 ]}  =     -18    &\quad
  {[  8,  4,  0, -8 ]}  =      21    \\
  {[  8,  4, -2, -6 ]}  =      68    &\quad
  {[  8,  2,  0, -6 ]}  =     -48    &\quad
  {[  8,  2, -2, -4 ]}  =    -126    &\quad
  {[  6,  4,  2, -8 ]}  =     -28    \\
  {[  6,  4,  0, -6 ]}  =    -154    &\quad
  {[  6,  4, -2, -4 ]}  =     252    &\quad
  {[  6,  2,  0, -4 ]}  =       0    &\quad
  {[  4,  2,  0, -2 ]}  =    1106    \\
  {[  8,  6, -4, -8 ]}  =     -10    &\quad
  {[  8,  4, -2, -8 ]}  =       9    &\quad
  {[  8,  4, -4, -6 ]}  =      13    &\quad
  {[  8,  2,  0, -8 ]}  =      13    \\
  {[  8,  2, -2, -6 ]}  =      -6    &\quad
  {[  8,  0, -2, -4 ]}  =    -105    &\quad
  {[  6,  4,  0, -8 ]}  =     -21    &\quad
  {[  6,  4, -2, -6 ]}  =      72    \\
  {[  6,  2,  0, -6 ]}  =    -218    &\quad
  {[  6,  2, -2, -4 ]}  =      84    &\quad
  {[  4,  2,  0, -4 ]}  =     553    &\quad
  {[  8,  6, -6, -8 ]}  =      -4    \\
  {[  8,  4, -4, -8 ]}  =      -6    &\quad
  {[  8,  2, -2, -8 ]}  =      20    &\quad
  {[  8,  2, -4, -6 ]}  =      12    &\quad
  {[  8,  0, -2, -6 ]}  =     -48    \\
  {[  6,  4, -2, -8 ]}  =      12    &\quad
  {[  6,  4, -4, -6 ]}  =      64    &\quad
  {[  6,  2,  0, -8 ]}  =     -48    &\quad
  {[  6,  2, -2, -6 ]}  =     -64    \\
  {[  6,  0, -2, -4 ]}  =     -84    &\quad
  {[  4,  2,  0, -6 ]}  =     -84    &\quad
  {[  4,  2, -2, -4 ]}  =     560    &\quad
  {[  8,  4, -6, -8 ]}  =     -10    \\
  {[  8,  2, -4, -8 ]}  =       9    &\quad
  {[  8,  0, -2, -8 ]}  =      13    &\quad
  {[  8,  0, -4, -6 ]}  =     -21    &\quad
  {[  6,  4, -4, -8 ]}  =      13    \\
  {[  6,  2, -2, -8 ]}  =      -6    &\quad
  {[  6,  2, -4, -6 ]}  =      72    &\quad
  {[  6,  0, -2, -6 ]}  =    -218    &\quad
  {[  4,  2,  0, -8 ]}  =    -105    \\
  {[  4,  2, -2, -6 ]}  =      84    &\quad
  {[  4,  0, -2, -4 ]}  =     553    &\quad
  {[  8,  2, -6, -8 ]}  =     -14    &\quad
  {[  8,  0, -4, -8 ]}  =      21    \\
  {[  8, -2, -4, -6 ]}  =     -28    &\quad
  {[  6,  4, -6, -8 ]}  =     -18    &\quad
  {[  6,  2, -4, -8 ]}  =      68    &\quad
  {[  6,  0, -2, -8 ]}  =     -48    \\
  {[  6,  0, -4, -6 ]}  =    -154    &\quad
  {[  4,  2, -2, -8 ]}  =    -126    &\quad
  {[  4,  2, -4, -6 ]}  =     252    &\quad
  {[  4,  0, -2, -6 ]}  =       0    \\
  {[  2,  0, -2, -4 ]}  =    1106    &\quad
  {[  8,  0, -6, -8 ]}  =     -14    &\quad
  {[  8, -2, -4, -8 ]}  =      28    &\quad
  {[  6,  2, -6, -8 ]}  =     -42    \\
  {[  6,  0, -4, -8 ]}  =     105    &\quad
  {[  6, -2, -4, -6 ]}  =    -420    &\quad
  {[  4,  2, -4, -8 ]}  =     175    &\quad
  {[  4,  0, -2, -8 ]}  =    -483    \\
  {[  4,  0, -4, -6 ]}  =      91    &\quad
  {[  2,  0, -2, -6 ]}  =    1106    &\quad
  {[  8, -2, -6, -8 ]}  =       0    &\quad
  {[  6,  0, -6, -8 ]}  =    -112    \\
  {[  6, -2, -4, -8 ]}  =     224    &\quad
  {[  4,  2, -6, -8 ]}  =      56    &\quad
  {[  4,  0, -4, -8 ]}  =     252    &\quad
  {[  4, -2, -4, -6 ]}  =   -2576    \\
  {[  2,  0, -2, -8 ]}  =   -1792    &\quad
  {[  2,  0, -4, -6 ]}  =    3864    &\quad
  \end{array}
  \]
  \end{center}
  \caption{Second quaternary algebra structure $g$ on $V(8)$}
  \label{structureconstants8part2}
  \end{table}

\end{document}